\documentclass[12pt]{amsart}

\usepackage[utf8]{inputenc}
\usepackage{geometry}
\usepackage{amsmath}
\usepackage{mathtools}
\usepackage{amssymb}
\usepackage{comment}
\usepackage{bbm}
\usepackage{cjhebrew}
\usepackage{mathrsfs}
\usepackage{tikz}
\usepackage{tikz-cd}
\usepackage{hyperref}
\usepackage{multirow}
\usepackage{adjustbox}
\usepackage{minibox}
\usepackage[english]{babel}
\usepackage{amsthm,xpatch}
\usepackage{subfiles}
\usepackage[new]{old-arrows}
\usepackage[table]{xcolor}
\usepackage{nicematrix}
\usepackage{pifont}
\usepackage{lmodern}
\usepackage{placeins}
\usepackage{csquotes}
\usepackage[margin={11mm,11mm}]{caption}
\usepackage{enumitem}
\usepackage{soul}
\usepackage{graphicx}
\usepackage{thmtools}

\usepackage{times}
\usepackage[euler-digits,euler-hat-accent]{eulervm}

\hypersetup{
	colorlinks=true,
	linktocpage=true,
	linkcolor=[RGB]{161,1,49},
	citecolor=[RGB]{10,136,169},      
	urlcolor=cyan,
}

\makeatletter
\xpatchcmd{\@thm}{\fontseries\mddefault\upshape}{}{}{} 
\makeatother

\makeatletter
\def\l@subsection{\@tocline{2}{0pt}{2pc}{5pc}{}}
\makeatother
\newcommand{\nocontentsline}[3]{}
\let\origcontentsline\addcontentsline
\newcommand\stoptoc{\let\addcontentsline\nocontentsline}
\newcommand\resumetoc{\let\addcontentsline\origcontentsline}

\newtheorem{theorem}{Theorem}[section]

\newtheorem{lemma}[theorem]{Lemma}
\newtheorem{proposition}[theorem]{Proposition}
\newtheorem{corollary}[theorem]{Corollary}

\theoremstyle{definition}
\newtheorem{definition}[theorem]{Definition}
\newtheorem{definition*}{Definition}
\newtheorem{remark}[theorem]{Remark}
\newtheorem{example}[theorem]{Example}

\newtheorem{question}[theorem]{Question}

\newcommand{\R}{\mathbb{R}}
\newcommand{\C}{\mathbb{C}}

\newcommand{\N}{\mathbb{N}}

\newcommand{\Z}{\mathbb{Z}}
\newcommand{\Q}{\mathbb{Q}}

\newcommand{\acts}{\curvearrowright}

\newcommand{\dual}[1]{{#1}^{\wedge}}
\newcommand{\bidual}[1]{{#1}^{\wedge \wedge}}

\newcommand{\vect}[1]{\mathbf{#1}}

\title{Expansivity of algebraic semigroup actions}
\author{Miguel Donoso-Echenique}
\address{University of M{\"u}nster, Einsteinstra{\ss}e 62, M{\"u}nster 48149, Germany}
\email{mdonosoe@uni-muenster.de,migueldonosoe@gmail.com}


\begin{document}
	
\maketitle
	
\begin{abstract}
		For a semigroup $S$ and a right $\Z[S]$-submodule $J\leq \Z[S]^n$, we study expansivity of the algebraic action of $S$ induced on the Pontryagin dual of $\Z[S]^n/J$. We completely determine the class of semigroups for which expansivity of this action is characterized by the triviality of $J^\perp$, in terms of the existence of a finite subset $K\subseteq S$ such that $S=KS$. This condition is satisfied in particular by every monoid, and more generally by every semigroup with a left unital convolution Banach algebra, in which case we are able to extend the characterization of expansivity in terms of an invertibility condition when $J$ is finitely generated. We also exhibit examples of non-unital semigroups satisfying the hypotheses of our results.
\end{abstract}

\setcounter{tocdepth}{2}	
\tableofcontents{}

\section*{Introduction}
\thispagestyle{empty}
Given a countable group  $G$ and a compact Hausdorff abelian group $X$, an action $G\acts X$ by continuous automorphisms is called an \textit{algebraic action}. There is a correspondence---via Pontryagin duality---between algebraic actions of a group $G$ and $\Z[G]$-modules, where $\Z[G]$ is the \textit{integral ring} of $G$. More precisely, every algebraic action $G\acts X$ induces a $\Z[G]$-module structure on the Pontryagin dual $\dual{X}$, and every discrete $\Z[G]$-module $M$ gives rise to an algebraic action $G\acts \dual{M}$. Of particular importance are algebraic actions arising from finitely generated $\Z[G]$-modules, namely, actions of $G$ by continuous group automorphisms of the metrizable space $\dual{(\Z[G]^n/J)}$, where $J\leq \Z[G]^n$ is a $\Z[G]$-submodule.

One dynamical property of such actions that has been a subject of interest for several decades now is \textit{expansivity}, i.e., the ability to distinguish any two different elements up to a fixed distance by simultaneously translating them with a suitable group element. This property is relevant, for instance, in the study of entropy of group actions \cite{bowen2011entropy,deninger2007expansive}. The research on expansivity of algebraic actions was initiated in \cite{kitchens1989automorphisms,schmidt1990} for $G$ free abelian, see \cite{schmidt_dynamical} for a thorough exposition. In \cite{einsiedler2001algebraic} the authors deal with more general groups and provide a characterization of expansivity of the action $G\acts \dual{(\Z[G]/J)}$ in terms of the triviality of the annihilator $J^\perp\leq \ell^\infty(S)$, viewing the ideal $J\leq \Z[G]$ as a subset of $\ell^1(S)$. When $J=(a)$ is the principal ideal generated by $a\in \Z[G]$, it is proven in \cite{deninger2007expansive} that $G\acts \dual{(\Z[G]/(a))}$ is expansive if and only if $a$ is invertible in $\ell^1(G)$. In \cite[\S 3]{chung2015homoclinic} the authors conclude an analogous result for finitely generated $\Z[G]$-submodules of $\Z[G]^n$, see also the exposition in \cite[Chapter 13]{KerrLi}.

Algebraic actions of semigroups---that is, actions of countable semigroups by endomorpshims of a compact Hausdorff abelian group---have been recently explored \cite{Bruce_Li_2026,BRUCE2024110263}. We will be interested in studying the action $S\acts \dual{(\Z[S]^n/J)}$, where $J\leq \Z[S]^n$ is a right $\Z[S]$-module. It will be shown in \S\ref{sec:algebraic_actions_basics} that the topological group $\dual{(\Z[S]^n/J)}$ is isomorphic to a subgroup $X_J$ of $(\mathbb{T}^S)^n$, which is metrizable, and the action $S\acts X_J$ arising from the restriction shift action $S\acts (\mathbb{T}^S)^n$ will be called the \textbf{algebraic action induced by} $J$. We aim to study, for these actions, the property of expansivity, which we formally define now.

\begin{restatable}{definition*}{expansivitydef}
		A continuous action $S\acts X$ on a compact metric space $(X,d)$ is called \textbf{expansive} if there exists $\epsilon>0$ (the \textbf{expansivity constant}) such that for every pair of distinct elements $x,y\in X$ we have that $\sup_{s\in S}d(s\cdot x,s\cdot y)\geq \epsilon$.
\end{restatable}

To our knowledge, there has been no research around expansivity of algebraic semigroup actions $S\acts X_J$. Some complications arise when we replace a group $G$ by a semigroup $S$. For instance, $S$ could have no identity element, which plays a role in expansivity. For instance, the shift action of the additive semigroup of positive natural numbers $\N_{+}$ on $\{0,1\}^{\N_{+}}$ is not expansive. Even further, it might happen that $\ell^1(S)$ is a non-unital Banach algebra---or might not even admit an approximate identity, as is the case of $\N_+$---which prevents us from speaking of invertibility and makes unclear how to extend the results from \cite{chung2015homoclinic,deninger2007expansive}. Another problem is that the language in which the results from \cite{chung2015homoclinic,deninger2007expansive,einsiedler2001algebraic} are framed involves the Banach $^*$-algebra structure of $\ell^1(G)$: see, for instance, \cite[Lemma 3.8]{chung2015homoclinic}, \cite[Theorem 3.2]{deninger2007expansive} and \cite[Theorem 8.2]{einsiedler2001algebraic}. For a general semigroup $S$, however, there is no clear involution on $\ell^1(S)$.

The main purpose of this article is to explore the extent to which the characterizations of expansivity of algebraic group actions, as well as the techniques used to prove them, hold up when we replace the acting group by a semigroup. Our first main result completely determines the class of semigroups for which expansivity of the algebraic action $S\acts X_J$ is characterized by the triviality of the annihilator of $J\leq \Z[S]^n$, viewing $J$ as a subset of $\ell^1(S)^n$. This characterization is analogous to the equivalence between (i) and (ii) in \cite[Theorem 8.1]{einsiedler2001algebraic} and that between (1) and (2) in \cite[Lemma 3.8]{chung2015homoclinic}.

\begin{restatable}{thm}{theoremA}\label{thm:meta_characterization}
	Let $S$ be a semigroup. The following statements are equivalent: 
	\begin{enumerate}
		\item[\textup{(i)}] There exists a finite subset $K\subseteq S$ such that $S=KS$.
		\item[\textup{(ii)}] For every right $\Z[S]$-submodule $J\leq \Z[S]^n$, the action $S\acts X_J$ is expansive if and only if $J^\perp=\{0\}$.
	\end{enumerate}
\end{restatable}
In the case where $\ell^1(S)$ is left-unital, our second main result extends the characterization in terms of invertibility conditions when $J$ is finitely generated. This stands in close analogy with the equivalence between (1) and (4) in \cite[Lemma 3.8]{chung2015homoclinic}, and follows the techniques therein.
\begin{restatable}{thm}{theoremB}\label{thm:expansivity_invertibility}
	Let $S$ be a semigroup such that $\ell^1(S)$ has a left identity. Then, given $\mathbf{A}\in \mathrm{M}_{k\times n}(\Z[S])$ and a left identity $\mathbf{I}\in \mathrm{M}_n(\ell^1(S))$, the following are equivalent.
	\begin{enumerate}
		\item[\textup{(i)}] The algebraic action $S\acts X_{\mathbf{A}\Z[S]^k}$ is expansive.
		\item[\textup{(ii)}] There exists $\mathbf{B}\in \mathrm{M}_n(\Z[S])$ that is right $\mathrm{Re}\{\mathbf{I}\}$-invertible in $\mathrm{M}_n(\ell^1(S))$ such that $\mathbf{B}\Z[S]^n\subseteq \mathbf{A}\Z[S]^k$.
	\end{enumerate}
	If $\ell^1(S)$ is additionally unital, the above conditions are equivalent to the following.
	\begin{enumerate}
		\item[\textup{(iii)}] There exists $\mathbf{B}\in \mathrm{M}_n(\Z[S])$ invertible in $\mathrm{M}_n(\ell^1(S))$ such that $\mathbf{B}\Z[S]^n\subseteq \mathbf{A}\Z[S]^k$.
	\end{enumerate}
\end{restatable}
Under the additional hypothesis of stable finiteness of $\ell^1(S)$, which is satisfied by every inverse semigroup with finitely many idempotents, we obtain the following corollary.

\begin{restatable}{coro}{corollaryC}\label{corollary:stably_finite}
	Let $S$ be a semigroup such that $\ell^1(S)$ is unital and stably finite. Then, for every $\mathbf{A}\in\mathrm{M}_n(\Z[S])$, the following are equivalent.
	\begin{enumerate}
		\item[\textup{(i)}] The algebraic action $S\acts X_{\mathbf{A}\Z[S]^n}$ is expansive.
		\item[\textup{(ii)}] The matrix $\mathbf{A}$ is invertible in $\mathrm{M}_n(\ell^1(S))$.
	\end{enumerate}
\end{restatable}

The article is organized as follows. In \S \ref{sec:preliminaries_notation} we introduce some preliminaries and notation. Section \S \ref{sec:algebraic_actions_basics} deals with the correspondence between algebraic actions of $S$ and right $\Z[S]$-modules arising from Pontryagin duality. The main section of the article is \S \ref{sec:expansivity}, where we prove Theorems \ref{thm:meta_characterization} and \ref{thm:expansivity_invertibility}. In \S \ref{sec:concrete_semigroups} we study the conditions on $S$ appearing in Theorem \ref{thm:meta_characterization}, Theorem \ref{thm:expansivity_invertibility} and Corollary \ref{corollary:stably_finite}. Since the conditions from theorems \ref{thm:meta_characterization} and \ref{thm:expansivity_invertibility} are immediately satisfied when $S$ is a monoid, in \S \ref{sec:concrete_examples} we study two concrete classes of semigroups that are not monoids but still satisfy these conditions, including Rees matrix semigroups.

\noindent
\textbf{Acknowledgements.}
The last stages of this work were funded by the Deutsche Forschungsgemeinschaft (DFG, German Research Foundation) under Germany's Excellence Strategy EXC 2044/2 –390685587, Mathematics Münster: Dynamics–Geometry–Structure.
The author is indebted to Tullio Ceccherini-Silberstein for suggesting the topic of research, fruitful discussions and constant encouragement. Thanks are also due to Santiago G. Rendel, Eduardo Silva and an anonymous reviewer for several comments that improved the exposition of this article, and to G. H. Esslamzadeh for correspondence on amenability of semigroup Banach algebras.

\section{Preliminaries and notation}\label{sec:preliminaries_notation}
\stoptoc
\subsection{Semigroups and actions}
A \textbf{semigroup} is a set $S$ together with an associative binary operation $S\times S\to S$, denoted by $(s,t)\mapsto st$. We shall only deal with countable semigroups in this article. A \textbf{monoid} is a semigroup $S$ admitting an identity element, that is, an element $1_S\in S$ such that $1_Ss=s1_S=s$ for all $s\in S$. Standard references regarding the algebraic theory of semigroups are \cite{Clifford_Preston_I,Clifford_Preston_II}.

Let $S$ be a countable semigroup. An \textbf{action} of $S$ on a set $X$, denoted by $S\acts X$, is a map $S\times X\to X$, $(s,x)\mapsto s\cdot x$, such that for all $s,t\in S$ and $x\in X$ we have $s\cdot (t\cdot x)=st\cdot x$. If $S$ is a monoid, we moreover require that $1_S\cdot x=x$ for all $x\in X$. Given an action $S\acts X$ and $x\in X$, we define the orbit $Sx\coloneq\{s\cdot x:s\in S\}$. If $X$ is a topological space, $S\acts X$ is an action and for each $s\in S$ the map $x\mapsto s\cdot x$ is continuous, we say that the action $S\acts X$ is \textbf{continuous}.

Given two actions $S\acts X$ and $S\acts Y$, a map $\phi\colon X\to Y$ is called $S$\textbf{-equivariant} if $\phi(s\cdot x)=s\cdot \phi(x)$ for all $s\in S$ and $x\in X$. A \textbf{conjugacy} between two continuous left actions $S\acts X$ and $S\acts Y$ is an $S$-equivariant homeomorphism $\phi\colon X\to Y$.

\subsection{Product spaces and matrices}\label{prel:product_matrices}
We introduce some notation regarding product spaces and matrix spaces. If $X$ is a set and $n\geq 2$, elements $\vect{x}=(x_1,\dots,x_n)\in X^n$ will be regarded as $1$-row matrices. If $X$ has the structure of an additive group, then for any $x\in X$ and $1\leq i\leq n$, the element $x\mathbf{e}_i\in X^n$ will denote the tuple $(x_1,\dots,x_n)$ with $x_i=x$ and $x_j=0$ for all $j\neq i$.

Matrices in $\mathrm{M}_{n\times k}(X)$ will be denoted with bold capital letters. Let $\vect{x}=(x_1,\dots,x_n)\in X^n$, $\vect{x}'=(x_1',\dots,x_k')\in X^k$ and $\mathbf{A}=(a_{ij})_{i,j}\in \mathrm{M}_{n\times k}(X)$. If $\star \colon X\times R\to X$ is a right $R$-module structure and $\ast\colon R\times X\to X$ is a left $R$-module structure, the expressions $\vect{x}\star \mathbf{A}\in X^k$ and $\mathbf{A}\ast \vect{x}'\in X^n$ are defined in the usual way, that is, the $j$-th coordinate of $\vect{x}\cdot \mathbf{A}\in X^k$ and the $i$-th coordinate of $\mathbf{A}\ast \vect{x}'$ are, respectively,
$$\sum_{i=1}^nx_i\star a_{ij}\quad \text{and}\quad \sum_{j=1}^{k}a_{ij}\ast x_j'.$$

\subsection{The semigroup integral ring and convolution algebra}
Let $S$ be a countable semigroup. In the forthcoming sections, we will deal with two Banach spaces associated with $S$. We consider
$$\ell^\infty(S)\coloneqq\left\{f\colon S\to \C\,\Big{|}\, \sup_{t\in S}|f(t)|<\infty\right\}\quad\text{with norm }\|f\|_{\infty}\coloneqq\sup_{t\in S}|f(t)|,$$
and
$$\ell^1(S)\coloneq\left\{a\colon S\to \C\,\bigg{|}\, \sum_{t\in S}|a(t)|<\infty\right\}\quad\text{with norm }\|a\|_{1}\coloneq\sum_{t\in S}|a(t)|.$$
As usual, for every $n\in \N$ we consider the product Banach spaces $\ell^\infty(S)^n$ and $\ell^1(S)^n$ with the corresponding norms defined as
$$\|\vect{f}\|_\infty\coloneq\max_{1\leq i\leq n}\|f_i\|_\infty \quad\text{and}\quad \|\vect{a}\|_1\coloneq\sum_{i=1}^n\|a_i\|_1$$
for all $\vect{f}=(f_1,\dots,f_n)\in \ell^\infty(S)^n$ and $\vect{a}=(a_1,\dots,a_n)\in \ell^1(S)^n$. The $\C$-bilinear map $\langle\cdot\,,\cdot\rangle\colon \ell^\infty(S)^n\times \ell^1(S)^n\to \C$ will denote the \textbf{canonical dual pairing}, so that, for all $\vect{f}=(f_1,\dots,f_n)\in \ell^\infty(S)$ and $\vect{a}=(a_1,\dots,a_n)\in\ell^1(S)$,
$$\langle \vect{f},\vect{a}\rangle \coloneq\sum_{i=1}^{n}\langle f_i,a_i\rangle=\sum_{i=1}^{n}\sum_{t\in S}f_i(t)a_i(t).$$

Given $F\subseteq S$, we write $\mathbf{1}_F\in \ell^\infty(S)$ to denote the characteristic function of $F$. More generally, if $f\in\ell^\infty(S)$, the element $f\mathbf{1}_F\in\ell^\infty(S)$ will denote the map that takes the value $f(s)$ at every $s\in F$ and $0$ elsewhere. In the particular case where $F=\{s\}$, we let $\delta_s\coloneq\mathbf{1}_{\{s\}}$.

Together with the \textbf{convolution product} $\ast\colon \ell^1(S)\times \ell^1(S)\to \ell^1(S)$ given by
$$(a\ast b)(s)\coloneq\sum_{\substack{r,t\in S\\rt=s}}a(r)b(t)\quad \text{for all }a,b\in \ell^1(S)\text{ and }s\in S,$$
the space $\ell^1(S)$ becomes a Banach algebra, called the \textbf{convolution Banach algebra} of $S$. We define the \textbf{integral ring} of $S$ to be the subring $\Z[S]\leq \ell^1(S)$ of finitely supported maps with values in $\Z$. From the fact that $\delta_s\ast\delta_t=\delta_{st}$ for all $s,t\in S$ we see that $\{\delta_s:s\in S\}$ is a subsemigroup of $\Z[S]$ isomorphic to $S$. We refer the reader to \cite{dales2010banach} for details on the Banach algebra $\ell^1(S)$.

Given $n\in \N$, $\mathrm{M}_n(\ell^1(S))$ is the Banach algebra of $n\times n$ matrices with coefficients in $\ell^1(S)$, with the usual matrix product and the norm
$$\|\mathbf{A}\|_1\coloneq\sum_{1\leq i,j\leq n}|a_{ij}|\quad\text{for all }\mathbf{A}=(a_{ij})_{i,j}\in\mathrm{M}_n(\ell^1(S)).$$

\subsection{Pontryagin duality}
Given a locally compact abelian group $X$, its \textbf{Pontryagin dual} $\dual{X}$ is the locally compact abelian group consisting of all continuous group homomorphisms $\varphi\colon X\to \mathbb{T}$ with pointwise multiplication and the topology of uniform convergence on compact subsets. The Pontryagin-Van Kampen duality theorem (see e.g., \cite[Theorem 24.2]{hewitt_abstract}) asserts that the map 
$$\Phi\colon X\to \bidual{X}\,;\, \Phi(x)(\varphi)=\varphi(x)$$
is an isomorphism of topological groups, i.e., an homeomorphic isomorphism of groups. The Pontryagin dual of a compact abelian group is a discrete abelian group, and vice-versa. A concrete example of this is $\dual{\Z}\simeq\mathbb{T}$ and $\bidual{\Z}\simeq\dual{\mathbb{T}}\simeq\Z$.
\resumetoc

\section{Algebraic semigroup actions and right $\Z[S]$-modules}\label{sec:algebraic_actions_basics}
We will be interested in the class of \textbf{algebraic} semigroup actions, that is, actions of semigroups by continuous endomorphisms of compact Hausdorff abelian groups. Pontryagin duality gives a correspondence between such actions and right $\Z[S]$-modules, which we proceed to describe now.

Let $S\acts X$ be an action by group endomorphisms of a compact Hausdorff abelian group $X$. Then we can define a right $\Z[S]$-module structure on the abelian group $\dual{X}$ by setting 
$$(\varphi\delta_s)(x)\coloneq\varphi(s\cdot x)\quad\text{for all }\varphi\in\dual{X},s\in S\text{ and }x\in X.$$
This is well-defined as a consequence of the fact that $x\mapsto s\cdot x$ is a continuous group homomorphism for all $s\in S$. Conversely, if $M$ is a right $\Z[S]$-module, we regard its underlying additive group structure as a discrete abelian group, so that $\dual{M}$ is a compact Hausdorff abelian group. An action $S\acts \dual{M}$ by group endomorphisms can be defined by setting
$$(s\cdot \varphi)(m)\coloneq\varphi(m \delta_s)\quad\text{for all }s\in S, \varphi\in \dual{M} \text{ and }m\in M.$$
This action is continuous, as the convergence in the topology of $\dual{M}$ is given by uniform convergence. It is not difficult to see that, under the above described procedure, an algebraic action $S\acts X$ is conjugate to the algebraic action $S\acts \bidual{X}$ induced by the right $\Z[S]$-module structure on $\dual{X}$ obtained from the action $S\acts X$. This yields the desired correspondence between algebraic actions of $S$ and right $\Z[S]$-modules.

We will restrict ourselves to studying algebraic actions arising from finitely generated right $\Z[S]$-modules.
To give a concrete description of the Pontryagin dual of a finitely generated right $\Z[S]$-module, we let $\pi\colon \R\to \mathbb{T}$ be the canonical projection map, and extend it to $\pi\colon \ell^\infty(S,\R)^n\to (\mathbb{T}^S)^n$ for all $n\in \N$ in the natural way. We start with the following lemma.

\begin{lemma}
	The product $\langle\cdot\,,\cdot \rangle\colon(\mathbb{T}^S)^n\times \Z[S]^n\to \mathbb{T}$ given by the formula $\langle \vect{x},\vect{a}\rangle\coloneq\pi(\langle \hat{\vect{x}},\vect{a}\rangle)$, where $\vect{x}\in (\mathbb{T}^S)^n$, $\vect{a}\in \Z[S]^n$ and $\hat{\vect{x}}\in \ell^\infty(S,\R)^n$ is such that $\pi(\hat{\vect{x}})=\vect{x}$, is well-defined.
\end{lemma}

\begin{proof}
	If $\hat{\vect{x}}=(\hat{x}_1,\dots,\hat{x}_n)$ and $\hat{\vect{x}}'=(\hat{x}_1',\dots,\hat{x}_n')$ are elements of $\ell^\infty(S,\R)^n$ such that $\pi(\hat{\vect{x}})=\pi(\hat{\vect{x}}')\in (\mathbb{T}^S)^n$, then $\hat{\vect{x}}-\hat{\vect{x}}'\in (\Z^S)^n$. Hence,
	$$\langle \hat{\vect{x}},\vect{a}\rangle-\langle \hat{\vect{x}}',\vect{a}\rangle=\sum_{i=1}^n\sum_{t\in S}(\hat{x}_i-\hat{x}'_i)(t)a_i(t)\in \Z$$
	for all $\vect{a}=(a_1,\dots,a_n)\in \Z[S]^n$, so that $\pi(\langle \hat{\vect{x}},\vect{a}\rangle)=\pi(\langle \hat{\vect{x}}',\vect{a}\rangle)$.
\end{proof}

We first analyze the algebraic action induced by a free right $\Z[S]$-module of finite rank. For each $n\in \N$, $\Z[S]^n$ gives rise to an algebraic action $S\acts (\Z[S]^n)^\wedge$ via the formula
$$(s\cdot \varphi)(a_1,\dots,a_n)=\varphi(a_1\ast \delta_s,\dots,a_n\ast \delta_s),$$
where $\varphi\in (\Z[S]^n)^\wedge$, $a_1,\dots,a_n\in\Z[S]$ and $s\in S$. The topological group isomorphism
$$(\Z[S]^n)^\wedge=\mathrm{Hom}\left(\bigoplus_{i=1}^{n}\bigoplus_{s\in S}\Z\,,\mathbb{T}\right)\simeq  \prod_{s\in S}\prod_{i=1}^{n}\mathrm{Hom}(\Z,\mathbb{T})\simeq(\mathbb{T}^S)^n$$ 
provided by the Pontryagin duality, gives a conjugacy between the action $S\acts (\Z[S]^n)^\wedge$ and the shift action $S\acts (\mathbb{T}^S)^n$ defined as 
$$(s\cdot \mathbf{x})(t)=(x_1(ts),\dots,x_n(ts))\quad \text{ for all }s,t\in S\text{ and }\mathbf{x}=(x_1,\dots,x_n)\in (\mathbb{T}^S)^n.$$ 
With slight abuse of notation, which comes from identifying $(\mathbb{T}^S)^n\simeq (\mathbb{T}^n)^S$, we will sometimes denote the element $(x_1(ts),\dots,x_n(ts))$ by $\mathbf{x}(ts)$, so that $(s\cdot \mathbf{x})(t)=\mathbf{x}(ts)$. We now give a detailed account of the conjugacy between $S\acts \dual{(\Z[S]^n)}$ and $S\acts (\mathbb{T}^S)^n$.

\begin{proposition}\label{prop:free_correspondence}
	The topological group isomorphism $\Phi\colon \dual{[\Z[S]^n]}\to (\mathbb{T}^S)^n$ defined by
	$$\Phi(\varphi)\coloneq (t\mapsto\varphi(\delta_{t}\mathbf{e}_1),\dots,t\mapsto\varphi(\delta_{t}\mathbf{e}_n))\quad \text{for all }\varphi\in \dual{(\Z[S]^n)}$$
	is a conjugacy with inverse given by $\Phi^{-1}(\vect{x})(\vect{a})=\langle \vect{x}, \vect{a}\rangle$ for all $\vect{x}\in (\mathbb{T}^S)^n$ and $\vect{a}\in \Z[S]^n$.
\end{proposition}

\begin{proof}
	We first check that $\Phi$ is bijective. Let $\varphi\in (\Z[S]^n)^{\wedge}$, $a_1,\dots,a_n\in \Z[S]$, and write $\Phi(\varphi)=(x_1,\dots,x_n)$ so that $x_i(t)=\varphi(\delta_t\mathbf{e}_1)$ for every $t\in S$ and $1\leq i\leq n$. Then,
	\begin{align*}
		\Phi^{-1}(\Phi(\varphi))(a_1,\dots,a_n)&=\sum_{i=1}^{n}\langle x_i, a_i\rangle=\sum_{i=1}^{n}\sum_{t\in S}x_i(t)a_i(t)\\
		&= \sum_{i=1}^{n}\sum_{t\in S}\varphi(\delta_t\mathbf{e}_i)a_i(t)=\sum_{i=1}^{n}\varphi\left(a_i\mathbf{e}_i\right)\\
		&=\varphi(a_1,\dots,a_n),
	\end{align*}
	so $\Phi^{-1}(\Phi(\varphi))=\varphi$. On the other hand, if $(x_1,\dots,x_n)\in (\mathbb{T}^S)^n$ and $s\in S$, then writing $\Phi(\Phi^{-1}(x_1,\dots,x_n))=(y_1,\dots,y_n)$ we have
	$$y_j(t)=\Phi^{-1}(x_1,\dots,x_n)(\delta_t\mathbf{e}_j)=\sum_{i=1}^n\langle x_i, \delta_t\mathbf{e}_j\rangle=\langle x_j,\delta_t\rangle=x_j(t)$$
	for all $t\in S$ and $1\leq j\leq n$, so $\Phi(\Phi^{-1}(x_1,\dots,x_n))=(x_1,\dots,x_n)$. Therefore $\Phi$ is a bijection with inverse $\Phi^{-1}$.
	
	Since the topology on $\dual{(\Z[S]^n)}$ is given by convergence on compact (hence finite) subsets of $\Z[S]^n$, it is readily checked that $\Phi$ is continuous. It is thus a homeomorphism, as the involved spaces are compact Hausdorff.
	
	Finally, to see that $\Phi$ is $S$-equivariant note that for every $\varphi\in\dual{(\Z[S]^n)}$ and $s,t\in S$, writing $\Phi(s\cdot \varphi)=(x_1,\dots,x_n)$ and $\Phi(\varphi)=(y_1,\dots,y_n)$ we find that
	$$x_j(t)=(s\cdot\varphi)(\delta_t\mathbf{e}_j)=\varphi((\delta_t\ast\delta_s)\mathbf{e}_j)= \varphi(\delta_{ts}\mathbf{e}_j)=y_j(ts)=(s\cdot y_j)(t).$$
	Thus $\Phi(s\cdot \varphi)=s\cdot \Phi(\varphi)$.
\end{proof}

\begin{corollary}\label{prop:correspondence_algebraic<->subshift}
	Let $J\leq \Z[S]^n$ be a right $\Z[S]$-submodule, and let
	$$X_J\coloneq\left\{\vect{x}\in(\mathbb{T}^S)^{n}:\langle \vect{x}, \vect{a}\rangle=0_{\mathbb{T}}\text{  for all }\vect{a}\in J\right\}.$$ 
	The action
	$S\acts \dual{(\Z[S]^n/J)}$ is conjugate to the restriction of the action $S\acts (\mathbb{T}^S)^n$ to $X_J$.
\end{corollary}

\begin{proof}
	Let $\Psi\colon \{\varphi\in \mathrm{Hom}(\Z[S]^n,\mathbb{T}):\varphi\rvert_J\equiv 0\}\to \mathrm{Hom}(\Z[S]^n/J,\mathbb{T})$ be the topological group isomorphism granted by the universal property of the quotient group. By considering the action of $S$ induced by Pontryagin duality upon the domain of $\Psi$, and the restriction of the action $S\acts \dual{(\Z[S]^n)}$ to the codomain of $\Psi$, it is clear that $\Psi$ defines a conjugacy, as
	\begin{align*}
		\Psi(s\cdot \varphi)(\mathbf{a}+J)&=(s\cdot \varphi)(\mathbf{a})=\varphi(\mathbf{a}\ast\delta_s)=\Psi(\varphi)(\mathbf{a}\ast\delta_s+J)\\
		&=\Psi(\varphi)((\mathbf{a}+J)\ast\delta_s)=(s\cdot \Psi(\varphi))(\mathbf{a}+J)
	\end{align*} 
	for all $\mathbf{a}\in\Z[S]^n$. Therefore, noting that for the conjugacy $\Phi$ from Proposition \ref{prop:free_correspondence} we have
	$$\Phi(\{\varphi\in \mathrm{Hom}(\Z[S]^n,\mathbb{T}):\varphi\rvert_J\equiv 0\})=\{\vect{x}\in(\mathbb{T}^S)^n:\langle \vect{x},\vect{a}\rangle=0\text{ for all }\vect{a}\in J\},$$
	it follows that $\Phi\circ \Psi^{-1}$ defines a conjugacy $\dual{(\Z[S]^n/J)}\to X_J$.	
\end{proof}

\begin{definition}
	Let $J\leq \Z[S]^n$ be a right $\Z[S]$-submodule, and let
	$$X_J\coloneq\left\{\vect{x}\in(\mathbb{T}^S)^{n}:\langle \vect{x}, \vect{a}\rangle=0_{\mathbb{T}}\text{  for all }\vect{a}\in J\right\}.$$
	The action $S\acts X_J$ given as the restriction of the action $S\acts (\mathbb{T}^S)^n$ to $X_J$ will be called the \textbf{algebraic action induced by} $J$.
\end{definition}

\section{Expansive algebraic semigroup actions}\label{sec:expansivity}
The purpose of this section is to study the property of expansivity for actions arising from finitely generated right $\Z[S]$-modules. We will work with a fixed metric compatible with the product topology on $(\mathbb{T}^S)^n$. Let $\rho$ be the standard metric on $\R/\Z$, that is, the one defined as follows for $\alpha_1,\alpha_2\in \R/\Z$:
$$\rho(\alpha_1,\alpha_2)\coloneq\min_{k\in \Z}|\hat{\alpha}_1-\hat{\alpha}_2+k|,$$
where $\hat{\alpha}_i\in\R$ is such that $\pi(\hat{\alpha}_i)=\alpha_i$ for $i=1,2$. Then, taking an enumeration $\{s_1,s_2,s_3,\dots\}$ of $S$, the formula
$$d(\mathbf{x},\mathbf{y})\coloneq\max_{1\leq j\leq n}\sum_{i\geq 1}\frac{\rho(x_j(s_i),y_j(s_i))}{2^{i}[1+\rho(x_j(s_i),y_j(s_i))]},$$
for $\mathbf{x}=(x_1,\dots,x_n)$ and $\mathbf{y}=(y_1,\dots,y_n)$, defines a metric on $(\mathbb{T}^S)^n$ bounded by $1$.

\subsection{A general characterization of expansivity}\label{subsec:meta_characterization}
Given a right $\Z[S]$-submodule $J\leq \Z[S]^n$, the \textbf{annihilator} of $J$ is defined by
$$J^\perp\coloneq\left\{\mathbf{f}\in\ell^\infty(S)^n:\langle \mathbf{f}, \mathbf{a}\rangle=0\text{ for all }\mathbf{a}\in J\right\}.$$

\begin{proposition}\label{prop:expansivity_matrix_general_version}
	Let $S$ be a semigroup such that $KS=S$ for some finite $K\subseteq S$. Then for every right $\Z[S]$-submodule $J\leq \Z[S]^n$, the following are equivalent. 
	\begin{enumerate}
		\item[\textup{(i)}] The action $S\acts X_J$ is expansive.
		\item[\textup{(ii)}] We have $J^\perp=\{0\}$.
		\item[\textup{(iii)}] There exists a finitely generated right $\Z[S]$-submodule $\omega\leq J$ with $\omega^\perp=\{0\}$.
	\end{enumerate}
\end{proposition}

\begin{proof}
	Let us prove that (i) implies (ii). It suffices to show that $J^\perp\cap \ell^\infty(S,\R)^n=\{0\}$. Let $\vect{f}\in \ell^\infty(S,\R)^n$ be such that $\langle \vect{f}, \vect{a}\rangle=0$ for all $\vect{a}\in J$. Then, $\langle \lambda\vect{f}, \vect{a}\rangle=0$ for every $\lambda\in\R$ and $\vect{a}\in J$, so Proposition \ref{prop:correspondence_algebraic<->subshift} allows us to conclude that $\pi(\lambda \vect{f})\in X_{J}$ for every $\lambda\in \R$. 
	
	By expansivity, there exists an open neighborhood $U$ of $\mathbf{0}\in X_J$ such that $\bigcap_{t\in S}t^{-1}U=\{0\}$ (take, for instance, $U=B(0,\epsilon)$ with $\epsilon>0$ the expansivity constant). Since the shift action $S\acts \ell^\infty(S,\R)^n$ is by isometries, and the map $\pi\colon \ell^\infty(S,\R^n)\to (\mathbb{T}^n)^S$ is continuous and $S$-equivariant, there exists $\lambda_0\in\R$ such that $S\pi(\lambda \vect{f})\subseteq U$ for all $0<\lambda<\lambda_0$. Hence,
	$$\pi(\lambda \vect{f})\in\bigcap_{t\in S}t^{-1}(S\pi(\lambda \vect{f}))\subseteq \bigcap_{t\in S}t^{-1}U=\{0\},$$
	which means $\pi(\lambda \vect{f})=0$ for all $0<\lambda< \lambda_0$. In other words, $\lambda \vect{f}\in \ell^{\infty}(S,\Z)^n$ for all $0<\lambda< \lambda_0$, so we must have $\vect{f}=0$. Thus, we have (ii) as desired.
	
	Now we prove that (ii) implies (iii). Let $K\subseteq S$ be a finite subset satisfying $S=KS$, and consider $\Omega=\{\omega\leq J:\omega\text{ is a finitely generated right }\Z[S]\text{-submodule}\}$ ordered by inclusion. Assume that $\omega^\perp\neq \{0\}$ for all $\omega\in \Omega$. Then, we can find for each $\omega\in\Omega$ a non-zero element $\vect{f}^\omega=(f^\omega_1,\dots,f^\omega_n)\in\omega^\perp$, and by renormalizing we may assume that $\|\vect{f}^\omega\|_\infty=1$. Moreover, since $S=KS$ and $s\cdot \vect{f}^\omega\in \omega^{\perp}$ for all $s\in S$, by replacing $\vect{f}^\omega$ with $s\cdot \vect{f}^\omega$ for a suitable $s\in S$, we may assume that $\|\vect{f}^\omega\rvert_K\|_\infty\geq 1/2$. By compactness of $([-1,1]^n)^S$ we may take a limit point $\vect{f}=(f_1,\dots,f_n)\in \ell^\infty(S)^n$ of the net $(\vect{f}^\omega)_{\omega\in \Omega}$, where $\Omega$ is ordered by inclusion. Note that it must be the case that $\|\vect{f}\rvert_K\|_\infty\geq 1/2$, so in particular $\vect{f}\neq 0$. However, since $\langle\vect{f}^\omega, \vect{a}\rangle=0$ for all $\vect{a}\in \omega$ and $\omega\in\Omega$, if $\vect{a}\in J$ and $\epsilon>0$,  then taking $\omega\in \Omega$ large enough so that $\vect{a}\in \omega$ and $\sum_{i=1}^{n}\|(f_i-f^{\omega}_i)\mathbf{1}_{\mathrm{supp}(a_i)}\|_\infty<\epsilon$, we have
	\begin{align*}
		|\langle\vect{f}, \vect{a}\rangle|=|\langle\vect{f}-\vect{f}^\omega, \vect{a}\rangle|\leq \sum_{i=1}^n\sum_{t\in \operatorname{supp}(a_i)}|(f_i-f_i^\omega)(t)a_i(t)|\leq \epsilon\|\vect{a}\|_1.
	\end{align*}
	Thus, $\langle\vect{f}, \vect{a}\rangle=0$ for all $\vect{a}\in J$, which implies $\vect{f}\in J^\perp=\{0\}$, a contradiction. 
	
	Let us see that (iii) implies (i). Since the metric $\rho$ is invariant under translations by elements of $\R/\Z$, we have that $d(\mathbf{x},\mathbf{y})=d(\mathbf{x}-\mathbf{y},0)$ for all $\mathbf{x},\mathbf{y}\in X_\omega$, and so it suffices to show that there exists $\epsilon>0$ such that, for an arbitrary non-zero element $\vect{x}=(x_1,\dots,x_n)$ of $X_\omega$, $\sup_{s\in S}d(s\cdot\mathbf{x},0)\geq \epsilon$. Let $\mathbf{x}$ be such an element. There exists a non-zero $\vect{f}=(f_1,\dots,f_n)\in ([-1/2,1/2]^S)^n$ with $\pi(\vect{f})=\vect{x}$, and such that
	$$\|\vect{f}\|_\infty=\sup_{s\in S}\max_{1\leq j\leq n}\rho(x_j(s),0_{\mathbb{T}}).$$ 
	Now, since $\omega^\perp=\{0\}$ and $\vect{f}$ is non-zero, we must have that $\langle\vect{f}, \vect{a}\rangle\neq 0$ for some $\vect{a}\in \omega$. Writing $\omega=\mathbf{A}\Z[S]^k$ for some $k\in\N$ and $\mathbf{A}=(a_{ij})_{i,j}\in\mathrm{M}_{n\times k}(\Z[S])$, we have $\vect{a}=\mathbf{A}\ast \vect{b}$ for some $\vect{b}=(b_1,\dots,b_k)\in\Z[S]^k$. Therefore,
	$$0\neq \langle\vect{f}, \vect{a}\rangle=\langle\vect{f}, \mathbf{A}\ast \vect{b}\rangle=\bigg{\langle}\vect{f},\mathbf{A}\ast \sum_{j=1}^{k}\sum_{t\in S}b_j(t)\delta_t\mathbf{e}_j\bigg{\rangle}=\sum_{j=1}^{k}\sum_{t\in S}b_j(t)\langle\vect{f},\mathbf{A}\ast \delta_t\mathbf{e}_j\rangle,$$
	so there must exist $t\in S$ and $1\leq j\leq k$ such that $\langle\vect{f}, \mathbf{A}\ast \delta_{t}\mathbf{e}_j\rangle\neq 0$. Since $\vect{x}\in X_\omega$, we have that $\pi(\langle\vect{f}, \vect{a}\rangle)=\langle \vect{x}, \vect{a}\rangle=0$ for all $\vect{a}\in \omega$, by Proposition \ref{prop:correspondence_algebraic<->subshift}. Hence, $\langle\vect{f},\vect{a}\rangle\in\Z$ for every $\vect{a}\in \omega$, and in particular $\langle\vect{f},\mathbf{A}\ast \delta_{t}\mathbf{e}_j\rangle\in \Z$, yielding $|\langle\vect{f},\mathbf{A}\ast \delta_{t}\mathbf{e}_j\rangle|\geq 1$.
	Therefore, using the fact that $\|a_{ij}\ast \delta_t\|_1\leq \|a_{ij}\|_1\|\delta_t\|_1=\|a_{ij}\|_1$, we get
	\begin{align*}
		1&\leq |\langle\vect{f},\mathbf{A}\ast \delta_{t}\mathbf{e}_j\rangle|= \left|\sum_{i=1}^{n}\langle f_i,a_{ij}\ast \delta_{t}\rangle\right|\leq  \sum_{i=1}^{n}\|f_i\|_\infty\|a_{ij}\ast \delta_t\|_1\leq \|\vect{f}\|_\infty\|\mathbf{A}\|_1.
	\end{align*}
	With this, since $KS=S$, we have
	\begin{align*}
		\sup_{s\in S}\sup_{t\in K}\max_{1\leq j\leq n}\rho((s\cdot x_j) (t),0)&=\sup_{s\in S}\sup_{t\in K}\max_{1\leq j\leq n}\rho(x_j (ts),0)\\
		&=\sup_{s\in S}\max_{1\leq j\leq n}\rho(x_j(s),0)=\|\vect{f}\|_\infty\geq \frac{1}{\|\mathbf{A}\|_1}.
	\end{align*}
	Taking an enumeration $\{s_1,s_2,\dots\}$ of $S$, there is some $r\geq 1$ such that $K\subseteq \{s_1,\dots,s_r\}$.
	Since $\rho$ is bounded by $1$, the metric $\rho/(1+\rho)$ is bounded below by $\rho/2$, we have
	$$d(\vect{y},\vect{y}')\geq \max_{1\leq j\leq n}\frac{\rho(y_j(t),y_j'(t))}{2^{r}(1+\rho(y_j(t),y_j'(t)))} \geq\frac{1}{2^{r+1}}\max_{1\leq j\leq n}\rho(y_j(t),y_j'(t))$$ 
	for all $\vect{y}=(y_1,\dots,y_n),\vect{y}'=(y_1',\dots,y_n')\in (\mathbb{T}^S)^n$ and $t\in K$, which implies
	$$\sup_{s\in S}d(s\cdot \vect{x},0)\geq\frac{1}{2^{r+1}}\sup_{s\in S}\sup_{t\in K}\max_{1\leq j\leq n}\rho((s\cdot x_j) (t),0)\geq\frac{1}{2^{r+1}\|\mathbf{A}\|_1}.$$
	Therefore $\epsilon\coloneq (2^{r+1}\|\mathbf{A}\|_1)^{-1}$ is an expansivity constant for $S\acts X_\omega$. Since $\omega\leq J$, by duality we have $X_J\subseteq X_\omega$, yielding the expansivity of $S\acts X_J$.
\end{proof}

It turns out that the existence of a finite $K\subseteq S$ with $S=KS$ is necessary for the equivalence between (i) and (ii) in Proposition \ref{prop:expansivity_matrix_general_version}. This is already implicitly contained in \cite[Proposition 4.2]{ceccherini2025topological} for the case where $n=1$ and $J=m\Z[S]$ for some $m\geq 2$.

\theoremA*

\begin{proof}
 It follows from Proposition \ref{prop:expansivity_matrix_general_version} 	that (i) implies (ii), so let us prove that (ii) implies (i). Assume (ii), and suppose that $KS\subsetneq S$ for every finite $K\subseteq S$. The algebraic semigroup action of $S$ upon the closed subgroup $X\leq \mathbb{T}^S$ given by
	$$X=\{x\in \mathbb{T}^S:x(s)\in\{0,1/2\}\text{ for all }s\in S\}$$
	cannot be expansive by \cite[Proposition 4.2]{ceccherini2025topological}, but we include a brief proof for completeness. For any given $\epsilon>0$, there exists a finite $K\subseteq S$ such that $d(x,y)<\epsilon$ whenever $x$ and $y$ coincide on $K$. Choose an element $s_K\in S-KS$, and let $x,y\in X$ be equal over $S-\{s_K\}$, and such that $x(s_K)\neq y(s_K)$. Then $ts\neq s_K$ for every $s\in S$ and $t\in K$, so
	$$(s\cdot x)(t)=x(ts)=y(ts)=(s\cdot y)(t).$$
	Hence, $s\cdot x$ and $s\cdot y$ coincide on $K$ for all $s\in S$, implying $d(s\cdot x,s\cdot y)<\epsilon$ for all $s\in S$. As this has been done for arbitrary $\epsilon>0$, we conclude that $\sup_{s\in S} d(s\cdot x,s\cdot y)=0$ but $x\neq y$, so that $S\acts X$ is not expansive.
	
	Observe that $x(s)\in\{0,1/2\}$ if and only if $\langle x, 2\delta_s\rangle=2x(s)=0$. Hence, we have that $X=\dual{[\Z[S]/2\Z[S]]}$. Thus, $X$ is the algebraic action associated with the right $\Z[S]$-module $J=2\Z[S]$. Now, if $f\in J^\perp$, then for every $t\in S$ we have
	$$0=\langle f, 2\delta_t\rangle=2\sum_{r\in S}f(r)\delta_t(r)=2f(t),$$
	so $f=0$. Hence, $J^\perp=\{0\}$, but $S\acts X_J$ is non-expansive.
\end{proof}

Taking into consideration Theorem \ref{thm:meta_characterization}, we introduce the following definition.

\begin{definition}
	A semigroup $S$ will be called \textbf{expansive} if there exists a finite subset $K\subseteq S$ such that $S=KS$.
\end{definition}

\subsection{Expansivity and invertibility}\label{subsec:invertibility}
The purpose of this subsection is to study expansivity of algebraic actions for a semigroup $S$ such that $\ell^1(S)$ admits a left identity. We start by observing that in this situation, $S$ always satisfies property (i) of Theorem \ref{thm:meta_characterization}.

\begin{lemma}\label{lemma:unital_integral_ring}
	If $\ell^1(S)$ has a left identity, then $S$ is expansive.
\end{lemma}

\begin{proof}
	Let $e\in\ell^1(S)$ be a left identity. Then $\mathrm{Re}\{e\}\in\ell^1(S)$ is a left identity too. Let $K\subseteq S$ be a finite subset such that $\|\mathrm{Re}\{e\}\mathbf{1}_{K^\text{c}}\|_1<1$. Note that, for every $s\in S$ we have
	$$1=\delta_s(s)=(\delta_s\ast \mathrm{Re}\{e\})(s)=\sum_{t: st=s}\mathrm{Re}\{e\}(t),$$
	and hence $\{t\in S:st=s\}\cap K\neq \varnothing$. Indeed, if it were not the case, we would have
	$$1=\sum_{t\in S:st=s}\mathrm{Re}\{e\}(t)\leq \sum_{t\in K^{\text{c}}}|\mathrm{Re}\{e\}(t)|=\|\mathrm{Re}\{e\}\mathbf{1}_{K^{\text{c}}}\|_1<1,$$
	which is absurd. Therefore, for every $s\in S$ there exists $t\in K$ such that $s=ts\in KS$, and we conclude that $S=KS$. 
\end{proof}

The characterizations in Proposition \ref{prop:expansivity_matrix_general_version} were inspired by the group-theoretical versions from \cite[Lemma 6.8]{schmidt_dynamical}, \cite[Theorem 8.1]{einsiedler2001algebraic} and \cite[Theorem 3.2]{deninger2007expansive}. Nonetheless, we closely followed the arguments in \cite[\S 3]{chung2015homoclinic}, where the authors provide a characterization of expansivity that involves the Banach $^*$-algebra structure of $\mathrm{M}_n(\ell^1(G))$ (see, e.g., \cite[Lemma 3.8]{chung2015homoclinic}). This language proves useful to obtain a characterization in terms of invertibility conditions in $\mathrm{M}_n(\ell^1(G))$, but is not available for actions of general semigroups. It suffices, however, to frame things in terms of the dual maps associated with the convolution product.

\begin{definition}
	The \textbf{dual convolution product} $\star\colon \ell^\infty(S)\times\ell^1(S)\to \ell^\infty(S)$ is given by
	$$(f\star a)(s)=\sum_{t\in S}f(ts)a(t)$$
	for all $f\in\ell^\infty(S)$ and $a\in\ell^1(S)$.
\end{definition}

We regard the dual convolution product as a product $\star\colon \ell^\infty(S)^n\times \mathrm{M}_n(\ell^1(S))\to \ell^\infty(S)^n$ for every $n\geq 1$, as explained in \S \ref{prel:product_matrices}.

\begin{proposition}\label{prop:Banach_module}
	We have $\langle \vect{f},\mathbf{A}\ast \vect{b}\rangle=\langle \vect{f}\star \mathbf{A},\vect{b}\rangle$ for all $\vect{f}\in \ell^\infty(S)^n$, $\mathbf{A}\in\mathrm{M}_{n\times k}(\ell^1(S))$ and $\vect{b}\in\ell^1(S)^m$. In particular, for all $n\in\N$ the product $\star \colon \ell^\infty(S)^n\times \mathrm{M}_n(\ell^1(S))\to \ell^\infty(S)^n$ is the dual right Banach $\mathrm{M}_n(\ell^1(S))$-module structure on $\ell^\infty(S)^n$, associated with the left Banach $\mathrm{M}_n(\ell^1(S))$-module structure $\ast \colon \mathrm{M}_n(\ell^1(S))\times \ell^1(S)^n\to \ell^1(S)^n$.
\end{proposition}

\begin{proof}
	For all $\vect{f}=(f_1,\dots,f_n)\in\ell^\infty(S)^n$, $\mathbf{A}=(a_{ij})_{i,j}\in\mathrm{M}_{n\times k}(\ell^1(S))$ and $\vect{b}=(b_1,\dots,b_k)\in \ell^1(S)^k$, it holds that
	\begin{align*}
		\langle\vect{f},\mathbf{A}\ast \vect{b}\rangle &= \sum_{j=1}^{k}\sum_{i=1}^{n}\langle f_i,a_{ij}\ast b_j\rangle = \sum_{j=1}^{k}\sum_{i=1}^{n} \sum_{t\in S}f_i(t)\sum_{rs=t}a_{ij}(r)b_j(s)\\
		&= \sum_{j=1}^{k}\sum_{i=1}^{n} \sum_{r,s\in S}f(rs)a_{ij}(r)b(s)=\sum_{j=1}^{k}\sum_{i=1}^{n} \sum_{s\in S}(f\star a_{ij})(s)b_j(s)\\
		&= \sum_{j=1}^{k} \bigg{\langle}\sum_{i=1}^{n}f_i\star a_{ij}, b_j\bigg{\rangle}=\langle \vect{f}\star \mathbf{A},\vect{b}\rangle.
	\end{align*}
	Hence the first claim is proved. The left $\mathrm{M}_n(\ell^1(S))$-module structure $\ast \colon \mathrm{M}_n(\ell^1(S))\times \ell^1(S)^n\to \ell^1(S)^n$ is determined by the maps $\psi_\mathbf{A}\colon \ell^1(S)^n\to \ell^1(S)^n$, $\vect{b}\mapsto \mathbf{A}\ast \vect{b}$. What we have just proven implies that for all $\mathbf{A}\in\mathrm{M}_n(\ell^1(S))$, the map $\varphi_\mathbf{A}\colon \ell^\infty(S)^n\to \ell^\infty(S)$, $\vect{f}\mapsto \vect{f}\star \mathbf{A}$ is exactly the dual map of $\psi_\mathbf{A}$. Hence, the maps $\varphi_\mathbf{A}$ determine a right $\mathrm{M}_n(\ell^1(S))$-module structure, dual to $\ast \colon \mathrm{M}_n(\ell^1(S))\times \ell^1(S)^n\to \ell^1(S)^n$. Finally, the fact that
	$$\|\vect{f}\star \mathbf{A}\|_\infty=\sup_{\|\vect{b}\|_1\leq 1}|\langle\vect{f}\star \mathbf{A},\vect{b}\rangle|=\sup_{\|\vect{b}\|_1\leq 1}|\langle\vect{f},\mathbf{A}\ast\vect{b}\rangle|\leq \sup_{\|\vect{b}\|_1\leq 1}\|\vect{f}\|_\infty\|\mathbf{A}\|_1\|\vect{b}\|_1\leq \|\vect{f}\|_\infty\|\mathbf{A}\|_1$$
	implies that said right $\mathrm{M}_n(\ell^1(S))$-module structure is Banach.
\end{proof}

\begin{proposition}\label{coro:expansivity_and_dual_convolution}
	Let $S$ be an expansive semigroup. Then, for every $\mathbf{A}\in\mathrm{M}_{n\times k}(\Z[S])$ the following are equivalent.
	\begin{enumerate}
		\item[\textup{(i)}] The action $S\acts X_{\mathbf{A}\Z[S]^k}$ is expansive.
		\item[\textup{(ii)}] The map $\varphi_\mathbf{A}\colon \ell^\infty(S)^n\to \ell^\infty(S)^k$ given by $\vect{f}\mapsto \vect{f}\star \mathbf{A}$ is injective.
	\end{enumerate}
\end{proposition}

\begin{proof}
	Note that, for every $\mathbf{A}\in\mathrm{M}_{n\times k}(\Z[S])$, we have $(\mathbf{A}\Z[S]^k)^\perp=\operatorname{ker}(\varphi_\mathbf{A})$. Indeed, we know from Proposition \ref{prop:Banach_module} that $\langle \vect{f},\mathbf{A}\ast \vect{b}\rangle=\langle \vect{f}\star \mathbf{A},\vect{b}\rangle$ for all $\vect{f}\in \ell^\infty(S)^n$ and $\vect{b}\in\ell^1(S)^m$. Hence, if $\vect{f}\star \mathbf{A}=0$, then $\langle \vect{f},\mathbf{A}\ast \vect{b}\rangle=0$ for all $\vect{b}\in\Z[S]^k$, so $\vect{f}\in (\mathbf{A}\Z[S]^k)^\perp$. On the other hand, if $\vect{f}\in (\mathbf{A}\Z[S]^k)^\perp$, we have that $\langle \vect{f}\star \mathbf{A},\vect{b}\rangle=0$ for all $\vect{b}\in\Z[S]^k$, which immediately implies $\vect{f}\star\mathbf{A}=0$. The equivalence between (i) and (ii) is now a direct consequence of Theorem \ref{thm:meta_characterization}.
\end{proof}
 
We will work with the following notion of invertibility in the case where $\ell^1(S)$ is left unital.

\begin{definition}
	Let $\mathfrak{A}$ be a Banach algebra. If $e\in \mathfrak{A}$ is a left identity, an element $a\in\mathfrak{A}$ will be called \textbf{right} $e$\textbf{-invertible} if there exists $b\in \mathfrak{A}$ such that $ab=e$.
\end{definition}

\begin{proposition}\label{thm:expansivity_criterion}
	Assume that $\ell^1(S)$ is left unital, and let $\mathbf{I}$ be a left identity in $\mathrm{M}_n(\ell^1(S))$. If $J\leq \Z[S]^n$ is a right $\Z[S]$-submodule and there exists an element $\mathbf{A}\in\mathrm{M}_n(\Z[S])$ that is right $\mathbf{I}$-invertible in $\mathrm{M}_n(\ell^1(S))$ such that $\mathbf{A}\Z[S]^n\leq J$, then the action $S\acts X_J$ is expansive.
\end{proposition}

\begin{proof}
	We first claim that $\vect{f}\star\mathbf{I}= \vect{f}$ for every $\vect{f}\in \ell^\infty(S)^n$. To show this, we first prove that for a left identity $e\in\ell^1(S)$, we have $f\star e= f$ for all $f\in\ell^\infty(S)$. Indeed, for all $r\in S$ it holds that $e\ast\delta_r= \delta_r$, so $(f\star e)\star \delta_r=f\star (e\ast \delta_r)=f\star \delta_r$. 
	By Lemma \ref{lemma:unital_integral_ring} $S$ is expansive, thus for every $s\in S$ we can write $s=rr'$ for some $r,r'\in S$. We then have
	\begin{align*}
		(f\star e-f)(s)&=((f\star e-f)\star \delta_r)(r')=(f\star\delta_r-f\star \delta_r)(r')=0,
	\end{align*}
	and we conclude that $f\star e=f$ as desired. Next, note that if $a\in\ell^1(S)$ is such that $a\ast b= 0$ for all $b\in \ell^1(S)$ then it holds that $f\star a= 0$ for all $f\in \ell^\infty(S)$. Indeed, for all $r\in S$ we have $a\ast \delta_r= 0$, implying
	$$(f\star a)(s)=((f\star a)\star\delta_r)(r')=(f\star (a\ast \delta_r))(r')=0.$$ 
	Finally, write $\mathbf{I}=(e_{ij})_{i,j}$. By considering the products $\mathbf{I}\ast \mathbf{A}_{ij}$, where $\mathbf{A}_{ij}$ is the matrix with all entries equal to zero except for the entry $(i,j)$ which is an arbitrary $a\in\ell^1(S)$, it is readily checked that $e_{jj}$ is a left identity in $\ell^1(S)$ for all $1\leq j\leq n$, and that $e_{ij}$ satisfies $e_{ij}\ast b= 0$ for every $b\in \ell^1(S)$ and $i\neq j$. Hence, for every $\vect{f}=(f_1,\dots,f_n)\in \ell^\infty(S)^n$ we get
	$$\mathbf{f}\star\mathbf{I}=\sum_{i=1}^n\sum_{j=1}^n(f_j\star e_{ij})\mathbf{e}_i=\sum_{i=1}^nf_i\mathbf{e}_i=\mathbf{f},$$
	and the claim is proven.
	
	The result now follows directly from the claim. Let $\vect{f}\in (\mathbf{A}\Z[S]^n)_\infty^\perp$ and $\mathbf{B}\in\mathrm{M}_n(\ell^1(S))$ such that $\mathbf{A}\ast \mathbf{B}=\mathbf{I}$. Then, as $\vect{f}\star\mathbf{I}= \vect{f}$, we have $\vect{f}=\vect{f}-(\vect{f}\star \mathbf{A})\star \mathbf{B}=\vect{f}-\vect{f}\star (\mathbf{A}\ast \mathbf{B}) =0$. Therefore $(\mathbf{A}\Z[S]^n)_\infty^\perp$ is trivial, and so is $J^\perp\subseteq (\mathbf{A}\Z[S]^n)^\perp$. Since $S$ is expansive, we conclude by Theorem \ref{thm:meta_characterization} that $S\acts X_J$ is an expansive action.
\end{proof}

\begin{remark}
	Proposition \ref{thm:expansivity_criterion} can be established for a semigroup $S$ such that $\ell^1(S)$ has a left approximate identity, in terms of the notion of approximate invertibility defined in \cite{esmeral2023approximately}. However, we will see in Proposition \ref{prop:KS=S+app_inv_implies_unital} that if $S$ is expansive and $\ell^1(S)$ admits a left approximate identity, then necessarily $\ell^1(S)$ has a left identity.
\end{remark}

\begin{lemma}\label{lemma:right_invertible}
	Let $\mathfrak{A}$ be a Banach algebra with a left identity $e$. Then, the set of right $e$-invertible elements of $\mathfrak{A}$ is open and non-empty.
\end{lemma}

\begin{proof}
	Note that $e$ is clearly right $e$-invertible, so that the set of right $e$-invertible elements is non-empty. Let $a\in \mathfrak{A}$ be right $e$-invertible, and let $a^{-1}\in \mathfrak{A}$ be such that $aa^{-1}=e$. 
	Given $x\in \mathfrak{A}$ with $\|x\|<\|a^{-1}\|^{-1}$, define $S_n \coloneq a^{-1}+\sum_{k=1}^{n}a^{-1}(xa^{-1})^{k}$	and note that $\|xa^{-1}\|< 1$, so the series $\sum_{k=1}^{\infty}\|a^{-1}(xa^{-1})^{k}\|$ converges. Since $\mathfrak{A}$ is a complete normed space, every absolutely convergent series converges (see \cite[Theorem 2.8]{bollobas1999linear}), so we can define $S_\infty :=a^{-1}+\sum_{k=1}^{\infty}a^{-1}(xa^{-1})^{k}$. Thus,
	$$(a-x)S_\infty=\lim_{n\to\infty}\left(e+\sum_{k=1}^{n}(xa^{-1})^k-\sum_{k=1}^{n+1}(xa^{-1})^k\right)=\lim_{n\to\infty}(e-(xa^{-1})^{n+1})=e,$$
	showing that $a-x$ is right $e$-invertible. Hence, every element at a distance less than $\|a^{-1}\|^{-1}$ from $a$ is right $e$-invertible, and we get that the set of right $e$-invertible elements is open.
\end{proof}

We now proceed to expand the characterization from Proposition \ref{prop:expansivity_matrix_general_version} for finitely generated $J$. Note that, for a given left identity $\mathbf{I}\in \mathrm{M}_n(\ell^1(S))$, the real part $\mathrm{Re}\{\mathbf{I}\}$ is also a (possibly different) left identity in $\mathrm{M}_n(\ell^1(S))$.

\theoremB*

\begin{proof}
	The fact that (ii) implies (i) directly follows from Proposition \ref{thm:expansivity_criterion}. 
		
	We prove now that (i) implies (ii). By Corollary \ref{coro:expansivity_and_dual_convolution}, the map $\varphi_\mathbf{A}\colon \ell^\infty(S)^n\to\ell^\infty(S)^k$ given by $\vect{f}\mapsto \vect{f}\star \mathbf{A}$ is injective. Consider now the $\C$-linear map $\psi_{\mathbf{A}}\colon \ell^1(S^1)^k\to \ell^1(S)^n$ given by $\vect{a}\mapsto\mathbf{A}\ast \vect{a}$. By Proposition \ref{prop:Banach_module}, we know that $\varphi_{\mathbf{A}}$ is the dual map of $\psi_\mathbf{A}$. Now, $\psi_\mathbf{A}$ must have dense image in $\ell^1(S)^n$, because otherwise there exists an element $\vect{a}\in \ell^1(S)^n-\overline{\mathrm{im}(\psi)}$, and by the Hahn-Banach Theorem we may extend the zero functional on $\overline{\mathrm{im}(\psi)}$ to obtain an element $\vect{f}\in \overline{\mathrm{im}(\psi_\mathbf{A})}^\perp\subseteq \ell^\infty(S)^n$ such that $\langle \vect{f},\vect{a}\rangle\neq 0$. Thus, $\vect{f}$ is non-zero but $\langle \vect{f},\psi_{\mathbf{A}}(\vect{b})\rangle=0$ for all $\vect{b}\in \ell^1(S)^k$, so that $\varphi_\mathbf{A}(\vect{f})=\psi_\mathbf{A}^*(\vect{f})=0$, which contradicts the injectivity of $\varphi_\mathbf{A}$.
	
	Let $\epsilon>0$ and let $\vect{e}_1,\dots,\vect{e}_n$ be the rows of $\mathbf{I}$. Since $\mathrm{im}(\psi_\mathbf{A})$ is dense in $\ell^1(S)^n$, we may choose elements $\vect{a}_1,\dots,\vect{a}_n \in \ell^1(S^1)^k$ such that for all $1\leq i\leq n$ we have $\|\psi_\mathbf{A}(\vect{a}_{i})-\vect{e}_i\|_1<\epsilon/n$. Observe that, for all $1\leq i\leq n$,
	\begin{align*}
		\|\psi_\mathbf{A}(\vect{a}_{i})-\vect{e}_i\|_1&= \|\psi_\mathbf{A}(\mathrm{Re}\{\vect{a}_{i}\})-\mathrm{Re}\{\vect{e}_i\}+\mathrm{i}\,(\psi_\mathbf{A}(\mathrm{Im}\{\vect{a}_{i}\})-\mathrm{Im}\{\vect{e}_i\})\|_1\\
		&\geq \|\psi_\mathbf{A}(\mathrm{Re}\{\vect{a}_{i}\})-\mathrm{Re}\{\vect{e}_i\}\|_1.
	\end{align*} 
	Hence, $\|\mathbf{A}\ast\mathrm{Re}\{\vect{a}_i\}-\mathrm{Re}\{\vect{e}_i\}\|_1<\epsilon/n$. Since $\Q[S]^k$ is dense in $\ell^1(S^1,\R)^k$ and $\psi_\mathbf{A}$ is continuous, we can choose elements $\vect{b}_{i}\in \Q[S]^k$ such that $\|\mathbf{A}\ast\vect{b}_i-\mathrm{Re}\{\vect{e}_i\}\|_1<\epsilon/n$ for all $1\leq i\leq n$. Hence, defining $\mathbf{B}_\epsilon\in M_{k\times n}(\Q[S])$ to be the matrix whose $i$-th column is $\vect{b}_i$, we get 
	$$\|\mathbf{A}\ast\mathbf{B}_\epsilon-\mathrm{Re}\{\mathbf{I}\}\|_1=\sum_{i=1}^{n}\|\mathbf{A}\ast\vect{b}_i-\mathrm{Re}\{\vect{e}_i\}\|_1<\epsilon.$$
	By Lemma \ref{lemma:right_invertible}, the set of right $\mathrm{Re}\{\mathbf{I}\}$-invertible elements is an open subset of $\mathrm{M}_n(\ell^1(S))$. Hence, taking $\epsilon$ sufficiently small we can ensure that $\mathbf{A}\ast\mathbf{B}_\epsilon$ is right $\mathrm{Re}\{\mathbf{I}\}$-invertible. Furthermore, since the rows of $\mathbf{B}_\epsilon$ are elements from $\Q[S]^k$, we may take $m\in \N$ such that $m\mathbf{B}_\epsilon\in \mathrm{M}_{k\times n}(\Z[S])$. Therefore, the matrix $\mathbf{B}:=\mathbf{A}\ast m\mathbf{B}_\epsilon$ belongs to $\mathrm{M}_n(\Z[S])$, is right $\mathrm{Re}\{\mathbf{I}\}$-invertible and satisfies $\mathbf{B}\Z[S]^n\subseteq \mathbf{A}\Z[S]^k$. 
	
	Finally, assume that $\ell^1(S)$ has an identity $e$. It is already clear that (iii) implies (ii). To see that (i) implies (iii), first note that for all $s\in S$ we have
	$$\delta_s\ast e=\delta_s\ast \mathrm{Re}\{e\}+\mathrm{i}(\delta_s\ast \mathrm{Im}\{e\})=\delta_s\in \ell^1(S,\R).$$
	Thus, $\delta_s\ast \mathrm{Im}\{e\}=0$ for all $s\in S$, which implies $\mathrm{Im}\{e\}=e\ast \mathrm{Im}\{e\}=0$. Hence $e=\mathrm{Re}\{e\}$, and we conclude that $\mathbf{I}=\mathrm{Re}\{\mathbf{I}\}$ as well. The same argument of the proof of the fact that (i) implies (ii) allows us to construct a matrix $\mathbf{B}\in\mathrm{M}_n(\Z[S])$ arbitrarily close to $\mathbf{I}$ such that $\mathbf{B}\Z[S]^n\subseteq \mathbf{A}\Z[S]^k$. Since the set of invertible elements of $\mathrm{M}_n(\ell^1(S))$ is open, we can assume that such $\mathbf{B}$ is invertible.
\end{proof}

\begin{definition}
	A unital ring $R$ is called \textbf{stably finite} if for every $n\geq 1$, every right invertible element from $\mathrm{M}_n(R)$ is also left invertible.
\end{definition}

\corollaryC*

\begin{proof}
	By the equivalence between (i) and (iii) from Theorem \ref{thm:expansivity_invertibility}, we just need to show that $\mathbf{A}$ is invertible in $\mathrm{M}_n(\ell^1(S))$ if and only if there exists $\mathbf{B}\in\mathrm{M}_n(\Z[S])$ that is invertible in $\mathrm{M}_n(\ell^1(S))$ satisfying $\mathbf{B}\Z[S]^n\subseteq \mathbf{A}\Z[S]^n$. Clearly, the invertibility of $\mathbf{A}$ implies this last condition. For the other direction, assume such $\mathbf{B}$ exists. It then follows that $\mathbf{B}\mathrm{M}_n(\ell^1(S))\subseteq \mathbf{A}\mathrm{M}_n(\ell^1(S))$ for every $n\geq 1$. Since the identity element $\mathbf{I}\in\mathrm{M}_n(\ell^1(S))$ belongs to $\mathbf{B}\mathrm{M}_n(\ell^1(S))$, there must exist $\mathbf{C}\in\mathrm{M}_n(\ell^1(S))$ such that $\mathbf{A}\ast\mathbf{C}=\mathbf{I}$, so $\mathbf{A}$ is right invertible in $\mathrm{M}_n(\ell^1(S))$. By the stable finiteness of $\ell^1(S)$ we conclude that $\mathbf{A}$ is invertible in $\mathrm{M}_n(\ell^1(S))$.
\end{proof}

\section{Expansive semigroups and their Banach algebras}\label{sec:concrete_semigroups}
In this section we study conditions on the semigroup $S$ appearing in Theorem \ref{thm:meta_characterization}, Theorem \ref{thm:expansivity_invertibility} and Corollary \ref{corollary:stably_finite}: expansivity, the existence of identities in $\ell^1(S)$ and stable finiteness of $\ell^1(S)$. We start by clarifying the relationship between these first two notions and pointing out some conditions that imply them. Then, we study stable finiteness in the context of inverse semigroups with finitely many idempotents.

\subsection{The relation between expansivity of $S$ and unitality of $\ell^1(S)$}
As we already saw in Lemma \ref{lemma:unital_integral_ring}, expansivity of $S$ is always satisfied when $\ell^1(S)$ has a left identity. The following proposition gives a converse for that result under an additional hypothesis. There, a \textbf{left approximate identity} in a Banach algebra $\mathfrak{A}$ is a net $(e_{\alpha})_{\alpha}$ in $\mathfrak{A}$ such that $e_\alpha a\to a$ for every $a\in \mathfrak{A}$.

\begin{proposition}\label{prop:KS=S+app_inv_implies_unital}
	Let $S$ be a semigroup. Then $\ell^1(S)$ admits a left identity if and only if $S$ is expansive and $\ell^1(S)$ admits a left approximate identity.
\end{proposition}

\begin{proof}
	It is shown in \cite[Proposition 4.3]{dales2010banach} that if $\ell^1(S)$ has a left approximate identity and $S=FS$ for some finite $F\subseteq E_S$, then $\ell^1(S)$ has a left identity. We will show that if $S$ is expansive and $\ell^1(S)$ admits a left approximate identity, then such $F$ exists. Let $(e_\alpha)_{\alpha}$ be a left approximate identity in $\ell^1(S)$. Note that for every $s\in S$ we have $e_\alpha\ast \delta_s\to \delta_s$, so that
	$$\sum_{r\in S: rs=s}e_\alpha(r)=(e_\alpha\ast \delta_s)(s)\to \delta_s(s)=1.$$
	In particular, for every $s\in S$ there exists $r\in S$ with $rs=s$. Let $K\subseteq S$ be a finite subset with $S=KS$. Choose, for every $t\in K$, an element $r_t\in S$ with $r_tt=t$, and let $K'=\{r_t:t\in K\}$. Then, for every $s\in S$ we can write $s=ts'$ with $t\in K$ and $s'\in S$, and it follows that $r_ts=r_tts'=ts'=s$. Hence $K'$ is finite, and for every $s\in S$ there is $t\in K'$ with $ts=s$ (in particular, $K'S=S$). 
	
	Let $n\coloneq |K'|$ and set $K'_n\coloneq K'$. If every element $t\in K'$ satisfies $t^2=t$, we are done by taking $F=K'_n$. Otherwise, if $t\in K'$ is not an idempotent, there is $t'\in K'$ with $t't=t$ and $t'\neq t$. It follows that every element $s\in S$ satisfying $ts=s$ also satisfies $t's=t'ts=ts=s$, so we can define $K_{n-1}'\coloneq K'_n-\{t\}$ and still for every $s\in S$ there is $t\in K'_{n-1}$ satisfying $ts=s$. By iterating this process, we see that for some $1\leq k\leq n$, $F\coloneq K'_k$ is a subset of $E_S$ and satisfies $S=FS$.
\end{proof}

\begin{example}\label{ex:LAI_not_LI}
	There are examples of semigroups $S$ such that $\ell^1(S)$ does not have a left identity but has a left approximate identity. Take the semigroup $\N_\wedge$ of natural numbers with the operation $n\wedge m=\min\{n,m\}$ for all $n,m\in \N$. The sequence $(\delta_n)_{n\in\N}$ is an approximate identity in $\ell^1(\N_\wedge)$, but there is no left identity element (see \cite[Example 4.10]{dales2010banach}).
\end{example}

In the case where $S$ is a \textbf{cancellative} semigroup, that is, if for all $s,t\in S$ we have $s=t$ whenever $sr=tr$ or $rs=rt$ for some $r\in S$, the relation between the hypotheses becomes stronger.

\begin{proposition}\label{prop:cancellative_case}
	Let $S$ be a cancellative semigroup. The following assertions are equivalent.
	\begin{enumerate}
		\item[\textup{(i)}] $S$ is expansive.
		\item[\textup{(ii)}] The algebra $\ell^1(S)$ has a left approximate identity.
		\item[\textup{(iii)}] The algebra $\ell^1(S)$ has a left identity.
		\item[\textup{(iv)}] $S$ is a monoid.
	\end{enumerate}
\end{proposition}

\begin{proof}
	The equivalence between (ii) and (iv) is proven in \cite[Corollary 1.3 and Remark 1.4]{gronbaek1988amenability}, and since (iv) implies (iii) and (iii) implies (ii), the equivalences between (ii), (iii) and (iv) follow. That (iii) implies (i) was proven in Lemma \ref{lemma:unital_integral_ring}, so it just remains to prove that (i) implies (iii).
	
	Let $K\subseteq S$ be finite such that $S=KS$, and note that if $t=t's$ for some $t,t'\in K$ with $t\neq t'$ and $s\in S$, then every $r\in S$ satisfying $r\in tS$ also satisfies $r\in tS=t'sS\subseteq t'S$. Hence, by removing such $t$ from $K$ and iterating the process, we may assume that every $t\in K$ satisfies $t=ts_t$ for some $s_t\in S$. With this, we get by left cancellativity that
	$$t=ts_t=t^2s_t\implies s_t=ts_t=t,$$
	so that $t^2=t$ for every $t\in K$. Now assume $|K|\geq 2$ and let $t_1,t_2\in K$ be different elements. Then $t_1t_2=t_1^2t_2=t_1t_2^2$, so left and right cancellativity yield that $t_1=t_1t_2=t_2$. Hence $|K|=1$, and we write $K=\{t\}$. Note that for every element $s\in S$ it holds that $s=ts'$ for some $s'\in S$, so $ts=tts'=ts'=s$ and we conclude that $t$ is a left identity element. Finally, by right cancellativity and the fact that $st=st^2$ we get $s=st$ for every $s\in S$, and we obtain that $S$ is a monoid, as desired.
\end{proof}

We now point out some conditions that imply expansivity or unitality of $\ell^1(S)$. Recall that a semigroup $S$ is called \textbf{regular} if for every $s\in S$ there exists $s^*\in S$ with $ss^*s=s$ and $s^*ss^*=s^*$. Also, an element $e\in S$ is called an \textbf{idempotent} if $e^2=e$, and we denote by $E_S$ the set of idempotent elements of $S$. When $S$ is regular, the elements $ss^*$ and $s^*s$ are idempotents for all $s\in S$, and thus $E_SS=SE_S=S$.

Expansivity of $S$, that is, the existence of a finite $K\subseteq S$ with $S=KS$, is immediately implied by any of the following conditions:
\begin{enumerate}
	\item[$\bullet$] $S$ has a left identity, or more generally if $\ell^1(S)$ has a left identity (Lemma \ref{lemma:unital_integral_ring})
	\item[$\bullet$] $S$ is regular and it either has a finite set of idempotents or is finitely generated.
\end{enumerate}
On the other hand, examples of non-expansive semigroups include $\N_\wedge$ from Example \ref{ex:LAI_not_LI} and any cancellative semigroup that is not a monoid, by Proposition \ref{prop:cancellative_case}. 

Moving to the existence of an identity in $\ell^1(S)$, it is clear that this is immediately satisfied by every monoid. We also include the following two examples.

\begin{example}[Semigroups with an amenable Banach algebra]
	Given a Banach algebra $\mathfrak{A}$, we denote by $\mathfrak{A}\,\hat{\otimes}\,\mathfrak{A}$ the completion of $\mathfrak{A}\otimes \mathfrak{A}$ with respect to the norm
	$$\|u\|=\inf\left\{\sum_{i=1}^{n}\|a_i\|\|b_i\|:u=\sum_{i=1}^{n}a_i\otimes b_i\right\},\quad u\in \mathfrak{A}\otimes \mathfrak{A},$$
	and call it the \textbf{projective tensor product} of $\mathfrak{A}$ and $\mathfrak{A}$. The space $\mathfrak{A}\,\hat{\otimes}\,\mathfrak{A}$ becomes a Banach $\mathfrak{A}$-bimodule via the formulas $a\cdot (u\otimes v)=(au)\otimes v$ and $(u\otimes v)\cdot a=u\otimes (va)$. Letting $\Delta_\mathfrak{A}\colon \mathfrak{A}\,\hat{\otimes}\, \mathfrak{A}\to \mathfrak{A}$ be the map given by $a\otimes b\mapsto ab$, a bounded net $(d_\alpha)_{\alpha}$ in $\mathfrak{A}\,\hat{\otimes}\,\mathfrak{A}$ is called an \textbf{approximate diagonal} if $\Delta_\mathfrak{A}d_\alpha$ is a right approximate identity in $\mathfrak{A}$ and $a\cdot d_\alpha-d_\alpha\cdot a\to 0$ for all $a\in \mathfrak{A}$. A Banach algebra $\mathfrak{A}$ is called \textbf{amenable} if it admits an approximate diagonal. For definitions and an overview of the theory of amenable Banach algebras we refer the reader to \cite{bonsall2012complete,runde2020amenable}. There is considerable literature regarding the amenability of semigroup Banach algebras \cite{dales2010banach,duncan1978amenability,paterson_duncan_amenability,gronbaek1988amenability,munn1955semigroup}. Amenability of the Banach algebra $\ell^1(S)$ imposes strong conditions upon the semigroup $S$: it is a regular semigroup with finitely many idempotents \cite[Theorem 2]{paterson_duncan_amenability}, and $\ell^1(S)$ is a unital algebra \cite[Corollary 10.6]{dales2010banach}.
\end{example}

\begin{example}[Inverse semigroups with finitely many idempotents]\label{ex:inverse_semigroups}
	A semigroup $S$ such that every $s\in S$ admits a \textit{unique} $s^*\in S$ with $ss^*s=s$ and $s^*ss^*=s^*$ is called an \textbf{inverse} semigroup. This is equivalent to saying that $S$ is a regular semigroup such that every two idempotents commute with each other (see, e.g., \cite[Theorem 1.17]{Clifford_Preston_I}). 
	
	The following observation is made in \cite[p. 6]{duncan1978amenability}. If $S$ is an inverse semigroup, then $E_S$ is a commutative subsemigroup of $S$. Thus, by \cite[Theorem 5.8]{hewitt19561} the algebra $\ell^1(E_S)$ is semisimple. If moreover $E_S$ is finite, then $\ell^1(E_S)$ is finite-dimensional and semisimple, hence
	isomorphic to a finite product of simple $\C$-algebras. By the Wedderburn-Artin Theorem, we conclude that $\ell^1(E_S)$ is isomorphic to a finite product $\prod_{i=1}^n\mathrm{M}_{n_i}(D_i)$, where each $D_i$ is a division $\C$-algebra. Hence, $\ell^1(E_S)$ has an identity $e\in\ell^1(E_S)$. Note that the associated element $e\mathbf{1}_{E_S}\in\ell^1(S)$ satisfies
	$$(e\mathbf{1}_{E_S}\ast \delta_r)(s)=\sum_{\substack{p\in E_S,t\in S\\tt'=s}}e(p)\delta_r(t)=\sum_{\substack{p\in E_S\\pr=s}}e(p)=(e\ast \delta_r)(s)\quad \text{ for all }p\in E_S\text{ and }s\in S,$$
	and therefore $e\mathbf{1}_{E_S}\ast\delta_s=(e\mathbf{1}_{E_S}\ast \delta_{ss^*})\ast \delta_{s}= \delta_{ss^*}\ast \delta_{s}= \delta_{s}$ for all $s\in S$.
	Hence, $e\mathbf{1}_{E_S}$ is an identity in $\ell^1(S)$, so $\ell^1(S)$ is unital. 
\end{example}

\subsection{Stable finiteness for inverse semigroups with finitely many idempotents}
Throughout this section, $S$ will be an inverse semigroup. If $|E_S|<\infty$, Example \ref{ex:inverse_semigroups} above shows that $\ell^1(S)$ is unital, so in particular $S$ is expansive by Lemma \ref{lemma:unital_integral_ring}. We want to see that in this case $\ell^1(S)$ is also stably finite, so that $S$ satisfies the hypothesis of Corollary \ref{corollary:stably_finite}.

I. Kaplansky observed that $\C[G]$ is stably finite for every $n\geq 1$ if $G$ is a group \cite[p. 122]{kaplansky1972fields}. This is a consequence of the fact that in this case the reduced group $C^*$-algebra $C^*_r(G)$ admits a faithful trace, and implies in fact that $\ell^1(G)$ is stably finite. We will follow the same proof strategy for $S$. We point out that if $|E_S|<\infty$ then the universal groupoid $\mathscr{G}_{\mathbf{u}}(S)$ associated with $S$ is Hausdorff (see, e.g., \cite[Theorem 4.20]{steinberg2010groupoid}), so by \cite[Corollary 7.9]{steinberg2026stable} and the paragraph thereafter, the reduced $C^*$-algebra $C_r^*(S)$ admits a faithful trace. Hence, the results in \cite{steinberg2026stable} provide a more general framework in which $\ell^1(S)$, if unital, is stably finite. Notwithstanding that, the proof under the assumption of finitely many idempotents becomes more concrete and direct, so we include it for completeness.

Recall that if $\mathfrak{A}$ is a $C^*$-algebra, a \textbf{faithful trace} is a positive linear map $\tau\colon \mathfrak{A}\to \C$ with $\tau(ab)=\tau(ba)$ for all $a,b\in\mathfrak{A}$, such that $\tau(a^*a)>0$ whenever $a\neq 0$. The following result is well-known.

\begin{proposition}
	Every unital $C^*$-algebra $\mathfrak{A}$ admitting a faithful trace is stably finite.
\end{proposition}

\begin{proof}
	If $\mathfrak{A}$ admits a faithful trace $\tau'$, then the trace $\tau\colon\mathrm{M}_n(\mathfrak{A})\to \C$ defined for every $\mathbf{A}=(a_{ij})_{i,j}\in \mathrm{M}_n(\mathfrak{A})$ by $\tau(\mathbf{A})=\sum_{j=1}^n\tau'(a_{jj})$ is faithful, since
	\[\tau(\mathbf{A}^*\mathbf{A})=\sum_{j=1}^n\tau'\left(\sum_{i=1}^{n}a_{ij}^*a_{ij}\right)=\sum_{1\leq i,j\leq n}\tau'(a_{ij}^* a_{ij})>0\]
	whenever $\mathbf{A}\neq \mathbf{0}$.
	
	The assertion that every right invertible element of $\mathrm{M}_n(\mathfrak{A})$ is also left invertible is equivalent to the following statement: for all projections $\mathbf{0}\leq \mathbf{Q}\leq \mathbf{P}$ admitting an element $\mathbf{V}\in\mathrm{M}_n(\mathfrak{A})$ with $\mathbf{P}=\mathbf{V}^*\mathbf{V}$ and $\mathbf{Q}=\mathbf{V}\mathbf{V}^*$, we have $\mathbf{P}=\mathbf{Q}$ \cite[Lemma 5.1.2]{rordam2000introduction}. If $\mathbf{P},\mathbf{Q},\mathbf{V}\in \mathrm{M}_n(\mathfrak{A})$ are as above, we get
	\[\tau(\mathbf{P}-\mathbf{Q})=\tau(\mathbf{V}^*\mathbf{V})-\tau(\mathbf{V}\mathbf{V}^*)=0,\]
	and since $\mathbf{P}-\mathbf{Q}$ is positive, we get that $\mathbf{P}=\mathbf{Q}$. Hence every right invertible element from $\mathrm{M}_n(\mathfrak{A})$ is invertible, and we conclude that $\mathfrak{A}$ is stably finite.
\end{proof}

We proceed to give the definition of the reduced $C^*$-algebra of $S$. Recall that $E_S$ is a commutative subsemigroup of $S$. We write $p\leq q$ if $p,q\in E_S$ are elements such that $pq=p$. Denote by $I(S)$ the inverse semigroup of all bijections $\sigma\colon A\to B$ between subsets $A,B\subseteq S$. The Vagner-Preston Theorem \cite[Proposition 2.1.3]{paterson_groupoids} provides an injective homomorphism of semigroups $\gamma\colon S\to I(S)$ by setting
\begin{align*}
	\gamma(s)\colon \{t\in S:tt^*\leq s^*s\}&\to \{t\in S:tt^*\leq ss^*\}\\
	t&\mapsto st.
\end{align*}
Define $\lambda(s)\in \mathcal{B}(\ell^2(S))$ so that $\lambda(s)(\delta_t)=\delta_{st}$ if $tt^*\leq s^*s$ and $0$ otherwise. We get a $^*$-homomorphism $\lambda\colon S\to \mathcal{B}(\ell^2(S))$, which by a theorem of Wordingham \cite{wordingham1982left} is a faithful representation of $S$. This extends to a $^*$-representation $\lambda\colon \ell^1(S)\to \mathcal{B}(\ell^2(S))$. 

\begin{definition}
	The \textbf{reduced} $C^*$\textbf{-algebra} $C_r^*(S)$ is the $\|\!\cdot\!\|$-closure of $\lambda(\ell^1(S))$ in $\mathcal{B}(\ell^2(S))$.
\end{definition}

\begin{proposition}
	If $S$ is an inverse semigroup and $E_S$ is finite, then $C^*_r(S)$ admits a faithful trace. In particular, $\ell^1(S)$ is stably finite.
\end{proposition}

\begin{proof}
	Let $\tau\colon C^*_r(S)\to \C$ be given by
	$$\tau(a)=\sum_{p\in E_S}\langle a(\delta_p),\delta_p\rangle_{\ell^2(S)}\quad \text{for all }a\in C^*_r(S).$$
	To prove that $\tau$ satisfies $\tau(ab)=\tau(ba)$ for all $a,b\in C^*_r(S)$, it suffices to show $\tau(\lambda(s)\lambda(t))=\tau(\lambda(t)\lambda(s))$ for all $s,t\in S$. Note that, by the definitions of $\lambda(s)$ and $\lambda(t)$, we have
	$$\tau(\lambda(s)\lambda(t))=\sum_{p\in E_S}\langle (\lambda(s)\lambda(t))(\delta_p),\delta_p\rangle_{\ell^2(S)}=\sum_{p:p\leq t^*t,tpt^*\leq s^*s}\delta_{stp}(p)=|\Omega_{s,t}|,$$
	where $\Omega_{s,t}:=\{p\in E_s:p\leq t^*t,tpt^*\leq s^*s,\text{ and }stp=p\}$. It is thus enough to show that $|\Omega_{s,t}|=|\Omega_{t,s}|$. Consider the map $\theta\colon \Omega_{s,t}\to \Omega_{t,s}$ given by $p\mapsto tpt^*$ for $p\in \Omega_{s,t}$. We first note that it is well-defined. Indeed, if $p\in \Omega_{s,t}$ we want to show that $\theta(p)\in E_S$ and satisfies $\theta(p)\leq s^*s$, $s\theta(p)s^*\leq t^*t$ and $ts\theta(p)=\theta(p)$. Note that, for such $p$, we have that
	\begin{itemize}
		\item $(tpt^*)^2=tp(t^*t)pt^*=tt^*tp^2t^*=tpt^*$, so $\theta(p)\in E_S$,
		\item the condition $tpt^*\leq s^*s$ by definition yields $\theta(p)\leq s^*s$,
		\item the fact that $stp=p$ implies $pt^*s^*=p$ because $p\in E_S$, and since $pt^*t=p$ we get
		\begin{align*}
			(s\theta(p)s^*)(t^*t)&=(stp)t^*s^*t^*t=(pt^*s^*)t^*t=pt^*t\\
			&=p=stp=st(pt^*s^*)=s\theta(p)s^*,
		\end{align*}
		and so $s\theta(p)s^*\leq t^*t$,
		\item and $ts\theta(p)=t(stp)t^*=tpt^*=\theta(p)$.
	\end{itemize}
	Hence $\theta$ is well-defined. Now, if $tpt^*=tqt^*$ for $p,q\in\Omega_{s,t}$, then
	$$p=t^*(tpt^*)t=t^*(tqt^*)t=q,$$
	so $\theta\colon \Omega_{s,t}\to\Omega_{t,s}$ is an injection. By symmetry, the map $\eta\colon \Omega_{t,s}\to\Omega_{s,t}$ given by $e\mapsto ses^*$ is an injection as well, so we have proven that $|\Omega_{s,t}|=|\Omega_{t,s}|$.
	
	Let us see that $\tau$ is faithful. For every $a\in C^*_r(S)$ the element $a^*a$ is self-adjoint with non-negative spectrum, and so the continuous functional calculus (see, e.g., \cite[Theorem VII.1]{reed1972methods}) gives a self-adjoint element $|a|\in C^*_r(S)$ such that $|a|^2=a^*a$. Also, if $s\in S$ then
	\[a^*a-|a|\lambda(s)^*\lambda(s)|a|=|a|(e-\lambda(s)^*\lambda(s))|a|=((e-\lambda(s)^*\lambda(s)|a|)^*(e-\lambda(s)^*\lambda(s))|a|,\]
	so the element $a^*a-|a|\lambda(s)^*\lambda(s)|a|$ is positive. Thus $\tau(|a|\lambda(s)^*\lambda(s)|a|)\leq \tau(a^*a)$. With this, if $\tau(a^*a)=0$ then for every $s\in S$ we have
	$$\sum_{p\in E_S}\|(a\lambda(s))(\delta_p)\|^2_{\ell^2(S)}=\tau(\lambda(s)^*a^*a\lambda(s))=\tau(|a|\lambda(s)^*\lambda(s)|a|)\leq\tau(a^*a)=0$$
	which implies $(a\lambda(s))(\delta_p)=0$ for all $p\in E_S$. Taking $p=s^*s$ we have $s^*s\leq s^*s$ so that $\lambda(s)(\delta_{s^*s})=\delta_{ss^*s}=\delta_s$, yielding $a(\delta_s)=0$ for all $s\in S$. Hence $\tau$ is faithful.
\end{proof}

\begin{example}[Bicyclic monoid]
	The \textbf{bicyclic monoid} is the free monoid $B$ generated by two elements $s,t$ subject to the relation $st=1_B$. This is an example of an inverse monoid, but the set $E_B$ of idempotents is infinite. As a monoid, $\ell^1(B)$ is unital, but $\ell^1(B)$ is \emph{not} stably finite, as $\delta_s\ast \delta_t=e$ but $\delta_{t}\ast \delta_s=\delta_{ts}\neq e$. One can see that for $\mathbf{A}=(\delta_s)\in \mathrm{M}_1(\Z[B])$, the assertion of Corollary \ref{corollary:stably_finite} fails. Indeed, since $\mathbf{A}$ is right invertible in $\mathrm{M}_1(\Z[B])$ then by Theorem \ref{thm:expansivity_invertibility} the action $B\acts X_{\mathbf{A}\Z[B]}$ is expansive. However, $\mathbf{A}$ cannot be invertible in $\mathrm{M}_1(\ell^1(B))$.
\end{example}

\begin{question}
	In \cite{ceccherini2025stable} the authors prove that for every surjunctive monoid $S$ and field $K$, the ring $K[S]$ is stably finite. Is it true that $\ell^1(S)$ is stably finite for this class of monoids?
\end{question}

\section{Concrete examples}\label{sec:concrete_examples}
In this final section, we study expansivity and unitality of $\ell^1(S)$ for two particular classes of semigroups.

\subsection{Rees matrix semigroups}
The following class of semigroups is of fundamental importance in the theory of semigroups. We refer the reader to \cite[Chapter 4]{dales2010banach} for an overview of this class with a special focus on their Banach algebras and concrete examples.

Let $G$ be a group and set $G^0:=G\cup\{0\}$. Let $I,\Lambda$ be finite sets and $\mathbf{P}=(p_{\lambda j})_{\lambda,j}\in \mathrm{M}_{\Lambda\times I}(G^0)$. Consider the set $I\times G^0\times \Lambda$ with the operation
$$(i,g,\lambda)(j,h,\mu)\coloneq (i,gp_{\lambda,j}h,\mu)\quad\text{for all }i,j\in I\text{, }g,h\in G^0\text{ and }\lambda,\mu\in \Lambda.$$
Define the \textbf{Rees matrix semigroup with sandwich matrix} $\mathbf{P}$ as the Rees quotient
$$\mathrm{M}^0(G;I,\Lambda;\mathbf{P})\coloneq (I\times G^0\times \Lambda)\big{/}\{(i,0,\lambda):i\in I,\lambda\in \Lambda\}.$$

\begin{proposition}
	Let $I,\Lambda$ be two finite sets, $G$ be a countable group and $\mathbf{P}\in\mathrm{M}_{\Lambda\times I}(G^0)$. Set $S=\mathrm{M}^0(G;I,\Lambda,\mathbf{P})$. We have the following.
	\begin{enumerate}
		\item[\textup{(i)}] The semigroup $S$ is expansive if and only if $\mathbf{P}$ has a non-zero entry.
		\item[\textup{(ii)}] There exists a finite $K\subseteq E_S$ with $S=KS$ if and only if every column of $\mathbf{P}$ contains a non-zero entry.
	\end{enumerate}
\end{proposition}

\begin{proof}
	Write $\mathbf{P}=(p_{\lambda i})_{i,\lambda}$. Let us prove (i). Since $S$ contains a non-zero element, it is clear that if $S=KS$ then $\mathbf{P}$ is non-zero. Conversely, let $i_0\in I$ and $\lambda_0\in \Lambda$ be such that $p_{\lambda_0 i_0}\in G$, and note that $K\coloneq\{(i,1_G,\lambda_0):i\in I\}$ is finite and satisfies
	$$(i,g,\lambda)=(i,1_G,\lambda_0)(i_0,p_{\lambda_0 i_0}^{-1}g,\lambda)$$
	for all $(i,g,\lambda)\in S$. Hence $S=KS$. 
	
	To see (ii), note that for $g\in G$ we have that $(i,g,\lambda)$ is an idempotent if and only if $p_{\lambda i}\in G$ and $g=p_{\lambda i}^{-1}$. Thus, if $S=KS$ for some $K\subseteq E_S$ then for all $i\in I$ there must be some $\lambda\in \Lambda$ with $p_{\lambda i}\in G$. Conversely, assume for each $i\in I$ there is $\lambda(i)\in\Lambda$ such that $p_{\lambda(i) i}\in G$. Then $K\coloneq\{(i,p_{\lambda(i)i}^{-1},\lambda(i)):i\in I\}\subseteq E_S$ is finite and
	$$(i,g,\lambda)=(i,p_{\lambda(i)i}^{-1},\lambda(i))(i,g,\lambda)\quad\text{for all }(i,g,\lambda)\in S,$$
	yielding $S=KS$
\end{proof}

The matrix $\mathbf{P}$ can be regarded as an element of $\mathrm{M}_{\Lambda\times I}(\ell^1(G))$ in a natural way by identifying $g\in G$ with $\delta_g\in\ell^1(G)$, and $0\in G^0$ with $0\in\ell^1(G)$. The results in the following proposition are well-known. See, e.g., pages 68 and 69 in \cite{dales2010banach}.

\begin{proposition}\label{prop:rees_matrix}
	The following statements hold.
	\begin{enumerate}
		\item[\textup{(i)}] The map $(a_{i\lambda})_{i,\lambda}\longmapsto \sum_{(g)_{i,\lambda}}a_{i\lambda}(g)\delta_{(g)_{i\lambda}}$ provides an isometric isomorphism of Banach algebras $\mathrm{M}_{I\times \Lambda}(\ell^1(G))\longrightarrow \ell^1(\mathrm{M}^0(G;I,\Lambda;\mathbf{P}))/\C\delta_0$ with multiplication in $\mathrm{M}_{I\times \Lambda}(\ell^1(G))$ given by $\mathbf{A}\cdot \mathbf{B}\coloneq\mathbf{A}\mathbf{P}\mathbf{B}$ for all $\mathbf{A},\mathbf{B}\in\mathrm{M}_{I\times \Lambda}(\ell^1(G))$. 			
		\item[\textup{(ii)}] The algebra $\ell^1(\mathrm{M}^0(G;I,\Lambda;\mathbf{P}))$ is unital if and only if $|\Lambda|=|I|$ and $\mathbf{P}$ is invertible as an element of $\mathrm{M}_{I\times \Lambda}(\ell^1(G))$.
	\end{enumerate}
\end{proposition}

\begin{remark}
	Examples 10.13, 10.14 and 10.15 in \cite{dales2010banach} provide matrices $\mathbf{P}\in \mathrm{M}_n(G^0)$ that are invertible in $\mathrm{M}_n(\ell^1(G))$, but such that $\mathbf{P}^{-1}\not\in \mathrm{M}_n(\Z[G])$. In these examples, the norm of the identity element of $\ell^1(S)$ belongs to $\Q-\Z$, so that $\Z[S]$ is a non-unital ring, but $\ell^1(S)$ is still a unital algebra.
\end{remark}

\subsection{Disjoint unions of semigroups}
This class of semigroups appears in \cite[Remark 4.7]{ceccherini2025topological}. Let $\Sigma_1,\dots,\Sigma_n$ be semigroups, $z$ be a symbol, and set
$$\amalg(\Sigma_1,\dots,\Sigma_n)\coloneqq\{z\}\sqcup \bigsqcup_{i=1}^n\Sigma_i.$$ 
We turn $\amalg(\Sigma_1,\dots,\Sigma_n)$ into a semigroup with the operation
$$st\coloneq \left\{\begin{array}{ll}
	st & \text{if } s,t\in \Sigma_i \text{ for some }1\leq i\leq n\\
	z & \text{otherwise}.
\end{array}\right.$$ 
Note that no element from $\amalg(\Sigma_1,\dots,\Sigma_n)$ can fix the elements of $\Sigma_1,\dots,\Sigma_n$ at the same time. Thus, $\amalg(\Sigma_1,\dots,\Sigma_n)$ cannot have a left identity.

\begin{proposition}
	Let $n\geq 2$, $\Sigma_1,\dots,\Sigma_n$ be semigroups, and set $S=\amalg(\Sigma_1,\dots,\Sigma_n)$. 
	\begin{enumerate}
		\item[\textup{(i)}] The semigroup $S$ is expansive if and only if each $\Sigma_i$ is expansive.
		\item[\textup{(ii)}] The algebra $\ell^1(S)$ has a left identity if and only if each $\ell^1(\Sigma_i)$ has a left identity. 
	\end{enumerate}
\end{proposition}

\begin{proof}
	To see (i), note that the ``only if'' part follows by taking $K_i:=K\cap \Sigma_i$ for each $1\leq i\leq n$, where $K\subseteq S$ is finite such that $S=KS$, and the ``if'' part by taking $K:=K_1\sqcup\dots\sqcup K_n$, where each $K_i\subseteq \Sigma_i$ is finite such that $\Sigma_i=K_i\Sigma_i$. 
	
	Now we prove (ii). Consider, for $1\leq i\leq n$, the character $\varphi_{\Sigma_i}\colon \ell^1(\Sigma_i)\to \C$ given by
	$$\varphi_{\Sigma_i}(a)\coloneq \sum_{t\in S}a(t)\quad\text{for all }a\in \ell^1(\Sigma_i).$$ 
	It can be verified that $\ell^1(S)\simeq \C\oplus \bigoplus_{i=1}^{n}\ell^1(\Sigma_i)$, where the product of $(\alpha,a_1,\dots,a_n)$ and $(\beta,b_1,\dots,b_n)$ is given in the first coordinate by
	\begin{equation}\label{eq:algebra_product}
		\sum_{i=1}^{n}\left[\varphi_{\Sigma_i}(b_i)\left(\alpha+\sum_{j\neq i}\varphi_{\Sigma_j}(a_j)\right)\right]+\beta\left(\alpha+\sum_{i=1}^{n}\varphi_{\Sigma_i}(a_i)\right),
	\end{equation}
	and defined coordinate-wise in the last $n$ coordinates. It is then clear that if $\ell^1(S)$ has a left identity $(\varepsilon,e_1,\dots,e_n)$, then each $e_i$ is a left identity in $\ell^1(\Sigma_i)$.
	
	Conversely, if for each $1\leq i\leq n$ there is a left identity $e_i\in\ell^1(\Sigma_i)$, then the element $e\coloneq (1-n,e_1,\dots,e_n)$ is a left identity in $\ell^1(S)$. Indeed, for all $(\beta,b_1,\dots,b_n)\in \ell^1(S)$ we have, by formula (\ref{eq:algebra_product}), that the first coordinate of the product $(1-n,e_1,\dots,e_n)\cdot (\beta,b_1,\dots,b_n)$ is
	\begin{align*}
		&\sum_{i=1}^{n}\left[\varphi_{\Sigma_i}(b_i)\left(1-n+\sum_{j\neq i}\varphi_{\Sigma_j}(e_j)\right)\right]+\beta\left(1-n+\sum_{i=1}^{n}\varphi_{\Sigma_i}(e_i)\right)\\
		&=\sum_{i=1}^{n}\varphi_{\Sigma_i}(b_i)\left(1-n+n-1\right)+\beta\left(1-n+n\right)=\beta,
	\end{align*}
	where we use the fact that $\varphi_{\Sigma_i}$ is a character, so $\varphi_{\Sigma_i}(e_i)=1$ for all $1\leq i\leq n$. This yields $(1-n,e_1,\dots,e_n)\cdot (\beta,b_1,\dots,b_n)=(\beta,b_1,\dots,b_n)$, hence $\ell^1(S)$ has a left identity.
\end{proof}

\bibliographystyle{alphaabbr}
\bibliography{ReferenciasAA}

\end{document}